\documentclass[12pt, reqno]{amsart}
\usepackage{amsmath, amsthm, amscd, amsfonts, amssymb, graphicx, color, mathrsfs}
\usepackage[bookmarksnumbered, colorlinks, plainpages]{hyperref}
\usepackage{cite}
\usepackage[all]{xy}
\usepackage{slashed}
\usepackage{tikz-cd}
\usepackage{mathabx}
\usepackage{tipa}
\usepackage{soul}
\usepackage{cancel}
\usepackage{ulem}
\usepackage{algorithm}
\usepackage{algpseudocode}

\newtheorem{theorem}{Theorem}[section]
\newtheorem{lemma}[theorem]{Lemma}

\newtheorem{proposition}[theorem]{Proposition}
\newtheorem{corollary}[theorem]{Corollary}

\theoremstyle{definition}
\newtheorem{definition}[theorem]{Definition}
\newtheorem{example}[theorem]{Example}

\newtheorem{algo}[theorem]{Algorithm}

\theoremstyle{remark}
\newtheorem{remark}[theorem]{Remark}
\numberwithin{equation}{section}

\begin{document}
\setcounter{page}{1}

\title[Control  of the Cauchy problem]{Control of the Cauchy problem on Hilbert spaces: A  global approach via symbol criteria}

\author[D. Cardona]{Duv\'an Cardona}
\address{
  Duv\'an Cardona:
  \endgraf
  Department of Mathematics: Analysis, Logic and Discrete Mathematics
  \endgraf
  Ghent University, Belgium
  \endgraf
  {\it E-mail address} {\rm duvan.cardonasanchez@ugent.be}
  }

\author[J. Delgado]{Julio Delgado}
\address{
  Julio Delgado:
  \endgraf
  Departamento de Matem\'aticas
  \endgraf
  Universidad del Valle
  \endgraf
  Cali-Colombia
    \endgraf
    {\it E-mail address} {\rm delgado.julio@correounivalle.edu.co}
  }

  \author[B. Grajales]{Brian Grajales}
\address{
  Brian Grajales:
  \endgraf
  Department of Mathematics
  \endgraf
  Universidad de Pamplona, Colombia
  \endgraf
  {\it E-mail address} {\rm brian.grajales@unipamplona.edu.co}
  }

\author[M. Ruzhansky]{Michael Ruzhansky}
\address{
  Michael Ruzhansky:
  \endgraf
  Department of Mathematics: Analysis, Logic and Discrete Mathematics
  \endgraf
  Ghent University, Belgium
  \endgraf
 and
  \endgraf
  School of Mathematics
  \endgraf
  Queen Mary University of London
  \endgraf
  United Kingdom
  \endgraf
  {\it E-mail address} {\rm michael.ruzhansky@ugent.be}
  }

\thanks{The authors were supported  by the FWO  Odysseus  1  grant  G.0H94.18N:  Analysis  and  Partial Differential Equations, by the Methusalem programme of the Ghent University Special Research Fund (BOF)
(Grant number 01M01021). B. Grajales has been partially supported by  Universidad de Pamplona. J. Delgado is also supported by Vice. Inv.Universidad del Valle Grant CI 71329, MathAmSud and Minciencias-Colombia under the project MATHAMSUD21-MATH-03. M. Ruzhansky is also supported  by EPSRC grant 
EP/R003025/2.
}


\begin{abstract} Let $A$ and $B$ be  invariant linear operators with respect to a decomposition $\{H_{j}\}_{j\in \mathbb{N}}$ of a Hilbert space $\mathcal{H}$  in subspaces of finite dimension. We give necessary and sufficient conditions for the controllability of the Cauchy problem $$ u_t=Au+Bv,\,\,u(0)=u_0,$$    in terms of the (global)  matrix-valued symbols $\sigma_A$ and $\sigma_B$ of $A$ and $B,$ respectively, associated to the decomposition   $\{H_{j}\}_{j\in \mathbb{N}}$.  Then, we present some applications  including the controllability of the Cauchy problem on compact manifolds for elliptic operators and the controllability  of fractional diffusion models for H\"ormander sub-Laplacians on compact Lie groups.  We also give conditions for the controllibility of wave and Schr\"odinger equations in these settings. 
\end{abstract} 

\maketitle

\tableofcontents
\allowdisplaybreaks

\section{Introduction}

\subsection{Outline and methodology} In this work we develop an approach to determine the controllability of the Cauchy problem on Hilbert spaces. Although there exists a very well-known criterion for the controllability of this problem, which is based on the Hilbert Uniqueness Method  due to J. L. Lions \cite{Lions88:I,Lions88:II}, which reduces the controllability of a system to the validity of the corresponding observability inequality for the adjoint system, here we provide a criterion inspired from the microlocal analysis of pseudo-differential operators (inspired in the notion of the symbol of an operator, see H\"ormander \cite{Hormander1985III}), which decouples the system 
\begin{equation}\label{Control:system:Intro}
  \frac{du}{dt}=Au+Bv(t),\,\,u(0)=u_0,\,\,  t\in [0,T], 
\end{equation} (here $A$ and $B$ are densely defined operators on a separable Hilbert space $\mathcal{H}$)
in an infinite number of finite-dimensional control systems
\begin{equation}\label{control:system:intro}
  {d\widehat{u}_\ell}/{dt}=A_\ell\widehat{u}_\ell+B_\ell\widehat{v}_\ell(t),\,\,\widehat{u}_\ell(0)=\widehat{u}_{0,\ell}\in H_\ell,\,\,t\in [0,T],\,\,\ell \in \mathbb{N}_0, 
\end{equation}
where one is allowed  to apply the {\it Kalman criterion}, see \cite{Kalman}. In this  context, Kalman's criterion says that the  {\it rank condition} 
\begin{equation}\label{Rank:intro}
    \textnormal{Rank}[B_\ell,A_\ell B_\ell,\cdots, A^{n_\ell-1}_\ell B_\ell]=n_\ell=\dim(H_\ell),
\end{equation} provides a necessary and sufficient condition for the controllability of \eqref{control:system:intro}. Our approach shows that the exact controllability of the system  \eqref{Control:system:Intro} implies the exact controllability of any system \eqref{control:system:intro}. On the other hand, our approach also shows that if every system \eqref{control:system:intro} is controllable and its controllability cost is uniformly bounded in $\ell$ (that is, if the coupled systems \eqref{control:system:intro} has a {\it globally finite controllability cost}) we are able to provide the exact controllability of \eqref{Control:system:Intro}. 

The decoupling procedure from \eqref{Control:system:Intro} to \eqref{control:system:intro} is carried out in such a way that the family of finite-dimensional subspaces $(H_\ell)_{\ell\in \mathbb{N}_0}$ provides an orthogonal decomposition of the underlying space  $\mathcal{H}=\bigoplus_{\ell} H_\ell.$ Relative to this decomposition, and to any choice of an orthonormal basis $\mathcal{B}_\ell$ of $H_\ell,$ there is a canonical Fourier transform
\begin{equation}
 \textbf{(FT):}\,\,  u\mapsto \widehat{u}(\ell)\in \mathbb{C}^{n_\ell},
\end{equation} provided by the orthogonal projections $P_{\ell}:\mathcal{H}\rightarrow H_{\ell},$ that in view of  the Fourier inversion formula
\begin{equation}
    u=\sum_{\ell\in \mathbb{N}_0}P_{\ell} u,
\end{equation}is given by
$ 
   \widehat{u}(\ell)=[P_{\ell} u]_{ \mathcal{B}_\ell}, 
$ that is the coordinate vector of $P_{\ell} u$ with respect to the basis $\mathcal{B}_\ell.$ Certainly one has the identification $\widehat{u}(\ell)\cong P_{\ell} u. $ In other words the decoupling procedure from \eqref{Control:system:Intro} into \eqref{control:system:intro} is nothing else that  taking the Fourier transform of the system \eqref{Control:system:Intro} relative to the decomposition $(H_\ell)_{\ell\in \mathbb{N}_0}.$ We explain it in the following diagram:
\begin{equation}
    \boxed{ \eqref{Control:system:Intro} \Longrightarrow \textbf{(FT)} \Longrightarrow \eqref{control:system:intro} }
\end{equation} On the other hand, a fact that is important for our further analysis  is the coupling procedure from all the systems in \eqref{control:system:intro} to \eqref{Control:system:Intro}. According to the standard terminology in {\it quantum mechanics} we do it by using the quantisation  procedure {\it symbol-to-operator}. In our case (and under the identification $\widehat{u}(\ell)\cong P_{\ell} u$) the sequences $(A_\ell)_{\ell\in \mathbb{N}}$ and $(B_\ell)_{\ell\in \mathbb{N}},$ are the symbols of the operators $A$ and  $B,$ respectively. The quantisation procedures
\begin{equation}
    (A_\ell)_{\ell\in \mathbb{N}}\mapsto A=\textbf{QP}((A_\ell)_{\ell\in \mathbb{N}}),\,\,(B_\ell)_{\ell\in \mathbb{N}}\mapsto B=\textbf{QP}((B_\ell)_{\ell\in \mathbb{N}}),\footnote{Here, we have employed the notation $A=\textbf{QP}((A_\ell)_{\ell\in \mathbb{N}})$ to indicate that the operator $A$ is the quantisation of the sequence $(A_\ell)_{\ell\in \mathbb{N}}.$ We have employed $\textbf{QP}$ to abbreviate {\it ``Quantisation procedure''}. In the standard terminology of the theory of pseudo-differential operators one also write $A=\textbf{Op}((A_\ell)_{\ell\in \mathbb{N}})$ to indicate that $A$ is the operator associated to the {\it symbol}  $(A_\ell)_{\ell\in \mathbb{N}}.$}
\end{equation}allow to analyse the properties of the operators $A$ and $B,$ from the properties of their symbols $\sigma_A(\ell):=A_\ell$ and $\sigma_B(\ell):=B_\ell,$ $\ell \in \mathbb{N},$ respectively. In other words, the coupling procedure from \eqref{control:system:intro} to \eqref{Control:system:Intro} is carried out by quantising the systems in \eqref{control:system:intro}. We explain it in the following diagram: 
\begin{equation}
    \boxed{ \eqref{control:system:intro} \Longrightarrow \textbf{(QP)} \Longrightarrow \eqref{Control:system:Intro} }
\end{equation}
As we will notice, the coupling and the decoupling procedure will be effective, in the sense that the information in the following diagram is preserved
\begin{equation}\label{Approach:Control}
    \boxed{ \eqref{Control:system:Intro} \Longrightarrow \textbf{(FT)} \Longrightarrow \eqref{control:system:intro} \Longrightarrow \textbf{(QP)} \Longrightarrow \eqref{Control:system:Intro} }
\end{equation}if the operators $A$ and $B$ leave invariant the orthogonal decomposition $(H_{\ell})_{\ell\in \mathbb{N}}.$ So, a fundamental {\it geometric property} assumed during this work is that every $H_{\ell}$ is an invariant subspace of $A$ and $B,$ respectively, that is, $$\forall \ell,\,\,\,AH_\ell\subset H_\ell \text{  and  }BH_\ell \subset H_{\ell}.$$ Having explained the methodology of our approach  we are going to explain our main result. We will also give several general examples of this property.
\subsection{The main result}
According to the theory of invariant operators on Hilbert spaces developed by the second and fourth author in \cite{FourierHilbert,DelRuzTrace1}, $A$ and $B$ are Fourier multipliers on $\mathcal{H}$ (associated to the decomposition $(H_{\ell})_{\ell\in \mathbb{N}}$).
The construction of the global matrix-valued symbol $\sigma_T$ in \cite{FourierHilbert,DelRuzTrace1} of a Fourier multiplier $T$ on a Hilbert space $\mathcal{H}$ can be found in Theorem \ref{THM:inv-rem} of Subsection \ref{Fourier:H}. With the notations employed in  Subsection \ref{Fourier:H} we present our main Theorem \ref{Main:Theorem:HS} in Section  \ref{Control:Compact}. In the case where $A$ generates a (strongly continuous semigroup) $C_0$-semigroup, our main theorem essentially says that 
\begin{equation}\label{Approach:Control}
 \small{\boxed{  \eqref{Control:system:Intro} \textbf{ is controllable } \Longleftrightarrow \forall \ell, \eqref{control:system:intro} \textbf{ satisfies the Kalman condition }  \eqref{Rank:intro} }}
\end{equation}and we compute in a sharp way the relation between the controllability cost of \eqref{Control:system:Intro}  with respect to the {\it global controllability cost} of the systems in \eqref{Rank:intro}, see Definition \ref{Definition:global:consts}. 
We refer the reader to Section   \ref{Control:Compact} for details. We observe that in the case where $M$ is a compact manifold without boundary, the Fourier analysis notion discussed above can be associated to any elliptic and positive elliptic pseudo-differential operator $E$ on $M$ in the sense of Seeley \cite{seeley1,seeley2}. Our approach includes this setting and other applications will be presented in the next subsection, see also Section \ref{Applications}.

\subsection{Applications}
From the mathematical perspective, there are many different contexts in which one can obtain an orthogonal decomposition of a Hilbert space $\mathcal{H}.$ For instance and for our purposes we give some examples:
\begin{itemize}
    \item On a compact manifold $M$ without boundary, $L^2(M)=\bigoplus_{\ell} H_\ell$ can be decomposed into the eigenspaces $H_\ell=\textnormal{Ker}(E-\lambda_\ell I)$ of a positive and elliptic pseudo-differential operator $E$ on $M.$
    \item On a compact manifold $M$ without boundary, (in particular,  on an arbitrary compact Lie group $M=G$) $L^2(M)=\bigoplus_{\ell} H_\ell$ can be decomposed into the eigenspaces $H_\ell=\textnormal{Ker}(\mathcal{L}^{s/2}-\lambda_\ell I)$ of a fractional power of a positive sub-Laplacian $\mathcal{L}=-\sum_{j=1}^kX_j^2,$ associated to a H\"ormander system of vector-fields $\mathcal{X}=\{X_1,\cdots, X_k\}$ satisfying the H\"ormander condition. It means that the vector fields in  $\mathcal{X}$ and their iterated commutators span the tangent space $TM.$
    \item In the case of $\mathbb{R}^n,$ $L^2(\mathbb{R}^n)=\bigoplus_{\ell} H_\ell$ can be decomposed into the eigenspaces $H_\ell=\textnormal{Ker}(\mathscr{H}-\lambda_\ell I)$ of the harmonic oscillator $\mathscr{H}=-\Delta_x+|x|^2,$ (or of more general anharmonic oscillators $\mathscr{H}_{l_1,l_2}=(-\Delta_x)^{l_1}+|x|^{2l_2}$, the  fractional relativistic Schr\"odinger operators, and of course the special case of  relativistic Schr\"odinger operators $\sqrt{I-\Delta}+|x|^{2l}$).
    \item On open bounded domains $\Omega$ of $\mathbb{R}^n,$ $L^2(\Omega)=\bigoplus_{\ell} H_\ell$ can be decomposed into the eigenspaces $H_\ell=\textnormal{Ker}((-\Delta)^s-\lambda_\ell I)$ of the spectral fractional Laplacian $(-\Delta)^s$ with homogeneous boundary Dirichlet data on $\partial \Omega.$
    \item In any separable complex Hilbert space $\mathcal{H}$ admitting an  unbounded self-adjoint operator $E$ with discrete spectrum (according to the spectral theorem), the previous situations supply  examples.
\end{itemize}Then, our analysis will include models of the form \eqref{Control:system:Intro} and then the following specific situations:
\begin{itemize}
    \item {\bf Control of fractional elliptic problems.} $M$ is a closed manifold, $\mathcal{H}=L^2(M),$ $A=E^s$ is a positive power of an elliptic operator $E$ on $M,$ and $B$ being an operator commuting with $E,$ (this condition assures that $B$ leaves invariant the eigenspaces of $A$).
    \item {\bf Control of fractional subelliptic problems.} Again, on a closed manifold $M,$ $A$ can be  a fractional power of a positive sub-Laplacian $\mathcal{L}=-\sum_{j=1}^kX_j^2,$ associated to a H\"ormander system of vector-fields $\mathcal{X}=\{X_1,\cdots, X_k\}$ satisfying the H\"ormander condition and $B$ being a continuous linear operator commuting with $A$.
    \item {\bf Control of fractional diffusion models for anharmonic operators and relativistic Schr\"odinger operators.} $A$ can be the harmonic oscillator  $\mathscr{H}=-\Delta_x+|x|^2,$ acting on $C^\infty_0(\mathbb{R}^n)\subset L^2(\mathbb{R}^n)$ (or $A$ can be a more general anharmonic oscillator of the type $\mathscr{H}_{l_1,l_2}=(-\Delta_x)^{l_1}+|x|^{2l_2}$) and $B$ commuting with $A.$ 
    Similarly, $A$ can be a relativistic Schr\"odinger operator 
 of the form $\sqrt{I-\Delta}+|x|^{2l}$ and $B$ commuting with $A.$
    \item {\bf Control in compact Lie groups setting.} $M=G$ is a compact Lie group and $A$ and $B$ are continuous linear operators on $C^\infty(G)$ being left-invariant. This mean that they commute with the left-action $L_x:f\mapsto f(x\cdot)$ of the group to $C^\infty(G).$
\end{itemize}Although the approach of this work, summarised in  \eqref{Approach:Control}, is designed for analysing general Cauchy problems on Hilbert spaces, our degree of generality is justified by a variety of applications (where particular contexts are given by the previous examples). Therein we allow the analysis of non-local models, namely, where the main term $A$ is a non-local operator as in the case of the fractional Laplacian $(-\Delta)^s$ on an open domain $\Omega$  or on a closed manifold $M,$ the fractional sub-Laplacian $(-\mathcal{L})^s,$  any positive power $E^s$ of an elliptic operator $E,$ or any PDE where $A$ having a discrete spectrum has principal terms involving pseudo-differential and/or integral terms.

\subsection{State-of-the-art} There has been a growing activity with respect to the research on the controllability of fractional diffusion models and other differential problems involving non-local operators. Next, we give some references related with this work.

\subsubsection{A general overview}
The growing research activity in the setting of fractional diffusion models is justified by  emerging models  in different branches of science and engineering. For instance, non-local and fractional equations appear as models in turbulence problems \cite{Bakunin2008}, in image processing \cite{Gilboa2008}, population dynamics \cite{DeRoos2002}, and e.g. in optimal control of fractional Laplacians with variable exponent models applied to image denoising \cite{Antil2019:2,Bahrouni2018}.   

In addition, several of the recent works in literature (see e.g. \cite{Acosta:etal:2017,Acosta2017:2,Andersson2019,Bonito2019,Boyer2011}) have been using numerical methods for  fractional Laplacians  where the techniques  are based in the works by Glowinski and J. L. Lions, see e.g. \cite{GlowinskiLions94,GlowinskiLions2008}. From the mathematical point of view, there has been a huge activity with recent pioneering works including extension techniques, see Caffarelli and Silvestre \cite{Caffarelli2007,Caffarelli2009}, and other models including inverse problems and analysis of non-local PDE, see e.g. S. Dipierro, X. Ros-Oton, and E. Valdinoci \cite{Dipierro2017,Dipierro2017:2},  Fall and  Felli \cite{Fall2014},  Ghosh,  R\"uland,   Salo,  and Uhlmann \cite{Ghosh2020}. For the recent activity involving  optimal control of diffusion models, we refer to \cite{Antil2019,Antil2019:2,Antil2019:3} and for the analysis of fractional hyperbolic and dispersive problems we refer to \cite{Biccari2019:2,Biccari2020,Biccari2020:2,FC:Zuazua2016}. Other recent works involving control and numerical methods  for fractional diffusion models can be found in the recent works  \cite{Biccari2019,Biccari2020:3,Claus2020} and the extensive list of references therein.

In the wide spectrum of the numerical analysis of non-local models, the {\it penalized uniqueness method} has been  used e.g. by Boyer, Hubert, and Rousseau in \cite{Boyer2011} and by  Glowinski and Lions in \cite{GlowinskiLions94} to compute control functions numerically in the analysis of fractional Laplacians. As we discussed above, these  operators are  of worthy interest since they appear in a large number of models describing practical situations. Among them we refer to \cite{Longhi2015} where a realization of the fractional Schr\"odinger equation is applied to optics, to \cite{WeissBloemenAntil2020} where the fractional Laplacian appears in the scalar Helmholtz equation which is used in electromagnetic interrogation in Earth's interior, and a classical paper by Mandelbrot and Van Ness \cite{MandelbrotNess1968} dealing with fractional Brownian motions. 

It is worth to mention that an alternative approach to study the controllability or even, the approximate controllability of the Cauchy problem $${du}/{dt}=Au+Bv(t),\,\,u(0)=u_0,\,\,  t\in [0,T] ,$$  is to accurately discretize  the operator $A.$ Some of these methods  have been developed in \cite{Acosta:etal:2017,Acosta2017:2,BiccariIMA,HuangOberman2014} for the fractional Laplacian and in \cite{NocheOtaSalg2015} for the {\it Dirichlet fractional Laplacian}, which is is defined as the power of
the Laplace operator, obtained by using the spectral decomposition of the Laplacian.

\subsubsection{Control theory on compact manifolds} The study of the controllability on Riemannian manifolds has a long tradition.  The fundamental models to be understood are the {\it heat } and the {\it wave } equation.  The exact controllability of the wave equation was proved firstly by
Chen and  Millman \cite{ChenMillman1980}. The exact controllability of the heat equation in the setting of internal control was proved in the seminal work of Lebeau and Robbiano \cite{LabeauRobbiano1995}. In particular, the method developed in \cite{LabeauRobbiano1995} changed the perspective on the field, by reducing the observability inequality of the adjoint system  to the validity of a spectral inequality, see Jerison and Lebeau \cite{JerisonLabeau} and Lebeau and Zuazua \cite{LebeauLebeau1998}.  Such a spectral inequality is a generalisation of the spectral inequality due to  Donnelly and Fefferman \cite{DonnellyFefferman83,DonnellyFefferman,DonnellyFefferman1990,DonnellyFefferman1992}. 

About fractional diffusion models for positive powers of general elliptic pseudo-differential operators, we refer the reader to  \cite{Cardona2022,CardonaDelgadoRuzhanskyControl2022}. On the other hand, we observe that in \cite{CoronFursikov1996}, the controllability in small time for the Navier-Stokes equations of incompressible fluids on compact two-dimensional manifolds was proved using purely analytic tools by Coron and Furkisov.

In the setting of the fractional heat equation  on open domains (with the suitable boundary conditions) several of these works have been dedicated to the internal controllability problem and the numerical analysis for this problems have been concentrated in the one dimensional case. To illustrate this, consider the fractional heat equation
    \begin{equation}\label{model:heat}
    \frac{du}{dt}=-(\Delta_M)^{s/2}u+1_{\omega}v,
    \end{equation}where $\omega\subset M$ is an open subset. For $s=2$ the null-controllability of the model \eqref{model:heat} was proved by  Lebeau and Robbiano \cite{LabeauRobbiano1995}. Micu and Zuazua \cite{MicuZuazua2006} and  
    Biccari and Hernández-Santamaría proved that, in one dimension, it is null controllable with a control function $v\in L^2(\omega\times(0,T))$ if and only if $1<s<2,$ and the authors have analysed the approximate controllability of the system for $0<s<1,$ see also \cite{BiccariIMA}. In \cite{Biccari2020:3}, Biccari, Warma, and Zuazua proved the same result with bounded control functions.

\subsection{Organisation of the work}

This paper is organised as follows:
 in Section \ref{Preliminaries} we present the preliminaries about the Fourier analysis on Hilbert spaces and the theory of invariant operators and their symbol properties as developed in \cite{FourierHilbert,DelRuzTrace1}.  We also present the construction of the matrix-valued symbols for continuous linear operators on compact Lie groups as developed in \cite{Ruz} and the results of the abstract control theory used in this work, namely, the Kalman condition and the observability criterion for the controllability of the Cauchy problem on Hilbert spaces.
     Our main result in the form of Theorem \ref{Main:Theorem:HS} will be presented in Section \ref{Control:Compact}.
     Section \ref{Applications} is dedicated to presenting a variety of applications of our main result.  It is organised as follows:
    In Subsection \ref{Control:compact:manifolds} we analyse the controllability of diffusion models for elliptic operators on compact manifolds.
         In Subsection \ref{Compact:lie:group:sec} we revisit Theorem   \ref{Main:Theorem:HS} in the context of a compact Lie group $G$ and  we give a criterion in Theorem \ref{theorem:compact:Lie:group} adapted to the Cauchy problem for general left-invariant operators on $G.$ The criterion is refined in terms of the matrix-valued symbols of the operators constructed from the group Fourier transform on $G,$ or equivalently, from the representation theory of the group as developed in \cite{Ruz}.
         Subsection \ref{section:applications:Hor} is dedicated to the controllability of fractional models determined by powers of H\"ormander sub-Laplacians on compact Lie groups.
         In Subsection \ref{Wave:vs:heat} we deduce the controllability of the heat operator from the controllability of the wave operator via a Kalman type analysis on each representation space. In Subsection \ref{Sch:section} we present the Kalman condition for  the control of the Schr\"odinger equation as an application of Theorem  \ref{Main:Theorem:HS}. Finally, in Section \ref{conclusions} we present some conclusions about the symbol criteria approach developed in this work.

\section{Fourier multipliers and abstract control theory}\label{Preliminaries}

\subsection{Fourier multipliers on Hilbert spaces}\label{Fourier:H}
We now recall the notion of invariant operators introduced in \cite{DelRuzTrace1} and which is based on the following theorem:
\begin{theorem}\label{THM:inv-rem}
Let $\mathcal{H}$ be a complex Hilbert space and let $\mathcal{H}^{\infty}\subset \mathcal{H}$ be a dense
linear subspace of $\mathcal{H}$. Let $\{d_{j}\}_{j\in\mathbb{N}_{0}}\subset\mathbb{N}$ and let
$\{e_{j}^{k}\}_{j\in\mathbb{N}_{0}, 1\leq k\leq d_{j}}$ be an
orthonormal basis of $\mathcal{H}$ such that
$e_{j}^{k}\in \mathcal{H}^{\infty}$ for all $j$ and $k$. Let $H_{j}:={\rm span} \{e_{j}^{k}\}_{k=1}^{d_{j}}$,
and let $P_{j}:\mathcal{H}\to H_{j}$ be the orthogonal projection.
For $f\in\mathcal{H}$, we denote $\widehat{f}(j,k):=(f,e_{j}^{k})_{\mathcal{H}}$ and let
$\widehat{f}(j)\in \mathbb{C}^{d_{j}}$ denote the column of $\widehat{f}(j,k)$, $1\leq k\leq d_{j}.$
Let $T:\mathcal{H}^{\infty}\to \mathcal{H}$ be a linear operator.
Then the following
conditions are equivalent:
\begin{itemize}
\item[(A)] For each $j\in\mathbb{N}_0$, we have $T(H_j)\subset H_j$. 
\item[(B)] For each $\ell\in\mathbb{N}_0$ there exists a matrix 
$\sigma_{T}(\ell)\in\mathbb{C}^{d_{\ell}\times d_{\ell}}$ such that for all $e_j^k$ 
$$
\widehat{Te_j^k}(\ell,m)=\sigma_{T}(\ell)_{mk}\delta_{j\ell}.
$$
\item[(C)]  For each $\ell\in\mathbb{N}_0 $ there exists a matrix 
$\sigma_{T}(\ell)\in\mathbb{C}^{d_{\ell}\times d_{\ell}}$ such that
 \[\widehat{Tf}(\ell)=\sigma_{T}(\ell)\widehat{f}(\ell)\]
 for all $f\in\mathcal{H}^{\infty}.$
\end{itemize}

The matrices $\sigma_{T}(\ell)$ in {\rm (B)} and {\rm (C)} coincide.

\noindent The equivalent properties {\rm (A)--(C)} follow from the condition 
\begin{itemize}
\item[(D)] For each $j\in\mathbb{N}_0$, we have
$TP_j=P_jT$ on $\mathcal{H}^{\infty}$.
\end{itemize}
If, in addition, $T$ extends to a bounded operator
$T\in{\mathscr L}(\mathcal{H})$ then {\rm (D)} is equivalent to {\rm (A)--(C)}.
\end{theorem} 
\begin{remark}
   Under the assumptions of Theorem \ref{THM:inv-rem}, we have the direct sum 
decomposition
\begin{equation}\label{EQ:sum}
\mathcal{H} = \bigoplus_{j=0}^{\infty} H_{j},\quad H_{j}={\rm span} \{e_{j}^{k}\}_{k=1}^{d_{j}},
\end{equation}
and we have $d_{j}=\dim H_{j}.$ 
\end{remark}
\begin{remark}In terms of the notation of Theorem \ref{THM:inv-rem}, for any $f\in \mathcal{H},$ the Fourier transform  
\begin{equation}
    \widehat{f}:\mathbb{N}_0\rightarrow\bigcup_{\ell \in \mathbb{N}_0}\mathbb{C}^{d_\ell\times d_\ell},\,\,\widehat{f}(j)=((f,e_{j}^{1})_{\mathcal{H}},\cdots, (f,e_{j}^{k})_{\mathcal{H}},\cdots,(f,e_{j}^{d_j})_{\mathcal{H}})^{T},
\end{equation} relative to the subspace  decomposition $\{H_j\}_{j\in \mathbb{N}_0
}$ admits the Fourier inversion formula
\begin{equation}\label{Fourier:Inversion:Formula}
    f=\sum_{\ell\in \mathbb{N}_0}(\widehat{f}(\ell),e_\ell)_{\mathbb{C}^{d_\ell}},\,\,
\end{equation}where $(x,y)\mapsto (x,y)_{\mathbb{C}^{d_\ell}}$ denotes the standard inner product on $\mathbb{C}^{d_\ell},$ and each   $e_\ell$ is the column vector
\begin{equation}
   e_\ell=(e^{1}_\ell,\cdots,e^{k}_\ell,\cdots, e^{d_\ell}_\ell )^{T}.
\end{equation}Note that the Plancherel formula takes the form
\begin{equation}\label{Plancherel:H}
 \forall f\in \mathcal{H},\,\,   \Vert f \Vert_{\mathcal{H}}^2=\sum_{\ell\in \mathbb{N}_0}\Vert \widehat{f}(\ell)\Vert^2_{\mathbb{C}^{d_\ell}}.
\end{equation}    
\end{remark}
\begin{remark}
    The two applications that we will consider will be with $\mathcal{H}=L^{2}(M)$ for a
compact manifold $M$ with $H_{j}$ being the eigenspaces of an elliptic classical
pseudo-differential operator $E$, or with $\mathcal{H}=L^{2}(G)$ for a compact Lie group
$G$ with $$ H_{j}=\textrm{span}\{\xi_{km}\}_{1\leq k,m\leq d_{\xi}}$$  for a
unitary irreducible representation $\xi\in[\xi_{j}]\in\widehat{G}$. The difference
is that in the first case we will have the eigenvalues of $E$ corresponding
to $H_{j}$'s are all distinct, while in the second case the eigenvalues of the Laplacian
on $G$ for which $H_{j}$'s are the eigenspaces, may coincide.
\end{remark}

\begin{definition}\label{Fourier:H}
    In view of properties (A) and (C), respectively, an operator $T$ satisfying any of
the equivalent properties (A)--(C) in
Theorem \ref{THM:inv-rem}, will be called an {\it  invariant operator}, or
a {\it  Fourier multiplier relative to the decomposition
$\{H_{j}\}_{j\in\mathbb{N}_{0}}$} in \eqref{EQ:sum}.
If the collection $\{H_{j}\}_{j\in\mathbb{N}_{0}}$
is fixed once and for all, we can just say that $T$ is {\it  invariant}
or a {\it  Fourier multiplier}. The family of matrices $\sigma$ will be called the {\it  matrix symbol of $T$ relative to the partition $\{H_{j}\}$ and to the
basis $\{e_{j}^{k}\}$}.
\end{definition}
\begin{remark}By following the notations in Definition \ref{Fourier:H}, in view of the Fourier inversion formula in \eqref{Fourier:Inversion:Formula}, we have the matrix-valued quantisation formula
 \begin{equation}\label{Quantisation:Formula}
    Tf=\sum_{\ell\in \mathbb{N}_0}(\sigma_T(\ell)\widehat{f}(\ell),e_\ell)_{\mathbb{C}^{d_\ell}},\,\,f\in \mathcal{H}^\infty.
\end{equation}   
\end{remark}
As a consequence of Theorem \ref{THM:inv-rem}, we have the following construction of global matrix-valued symbols on compact manifolds without boundary.
\begin{theorem}[Fourier multipliers on compact manifolds] Let $M$ be a closed manifold. Consider $E\in \Psi^\nu_{cl,+}(M)$ be a classical positive pseudo-differential operator of order $\nu>0$ on $M.$  Let
\begin{itemize}
    \item $H_{j}=\textnormal{Ker}(E-\lambda_jI)$ be the family of eigenspaces of $E$,
    \item $P_{j}:L^2(M)\to H_{j}$ be the corresponding orthogonal projections,
    \item  $\{d_{j}\}_{j\in\mathbb{N}_{0}}\subset\mathbb{N}$ be the sequence formed by the dimensions of each $H_j,$
    \item and assume that
$\mathcal{B}=\{e_{j}^{k}\}_{j\in\mathbb{N}_{0}, 1\leq k\leq d_{j}}$ is an
orthonormal basis of $L^2(M),$ where each $H_j$ is spanned by the basis ${\rm span} \{e_{j}^{k}\}_{k=1}^{d_{j}}.$
\end{itemize}
For $f\in L^2(M)$, we denote by $\widehat{f}(j,k):=(f,e_{j}^{k})_{\mathcal{H}}$ the Fourier coefficients of $f$ relative to the basis $\mathcal{B}.$  Let
$$ \widehat{f}(j)\in \mathbb{C}^{d_{j}}$$  denote the column of $\widehat{f}(j,k)$, $1\leq k\leq d_{j}.$
Let $T:C^{\infty}(M)\to L^2(M)$ be a  linear operator.
Then the following
conditions are equivalent:
\begin{itemize}
\item[(A)] For each $j\in\mathbb{N}_0$, we have $T(H_j)\subset H_j$. 
\item[(B)] For each $\ell\in\mathbb{N}_0$ there exists a matrix 
$\sigma_{T}(\ell)\in\mathbb{C}^{d_{\ell}\times d_{\ell}}$ such that for all $e_j^k$ 
$$
\widehat{Te_j^k}(\ell,m)=\sigma_{T}(\ell)_{mk}\delta_{j\ell}.
$$
\item[(C)]  For each $\ell\in\mathbb{N}_0 $ there exists a matrix 
$\sigma_{T}(\ell)\in\mathbb{C}^{d_{\ell}\times d_{\ell}}$ such that
 \[\widehat{Tf}(\ell)=\sigma_{T}(\ell)\widehat{f}(\ell)\]
 for all $f\in C^{\infty}(M).$
\end{itemize}

The matrices $\sigma_{T}(\ell)$ in {\rm (B)} and {\rm (C)} coincide.

\noindent The equivalent properties {\rm (A)--(C)} follow from the condition 
\begin{itemize}
\item[(D)] For each $j\in\mathbb{N}_0$, we have
$TP_j=P_jT$ on $\mathcal{H}^{\infty}$.
\end{itemize}
If, in addition, $T$ extends to a bounded operator
$T\in{\mathscr L}(L^2(M))$ then {\rm (D)} is equivalent to {\rm (A)--(C)}.
\end{theorem}

\begin{remark} Let $A,B:\mathcal{H}^\infty\rightarrow\mathcal{H}$ be Fourier multipliers. Assume that $A$ is the generator of a $C_0$-semigroup $S(t),$ that is
\begin{equation}
    \forall v\in \mathcal{H}^\infty,\,\,Av=\lim_{t\rightarrow 0}\frac{1}{t}\left(S(t)v-v\right),
\end{equation}where the limit is taken with respect to the norm on  $\mathcal{H}$. In this case one has that
    \begin{equation}\label{symbolproperties:b}
        \forall \ell \in \mathbb{N},\,\,\,\sigma_{S(t)}(\ell)=e^{t\sigma_A(\ell)} \textnormal{ and }\sigma_{B^\ast}(\ell)=\sigma_{B}(\ell)^*.
    \end{equation}
\end{remark}

\subsection{Abstract control theory} In this section we will present some results about the controllability of a general control system of the form
\begin{equation}\label{Maingeneral}
    \displaystyle\frac{du}{dt}=Au+Bv(t),\ t\in[0,T],
\end{equation}
where $A$ and $B$ are continuous linear $\mathcal{H}$-valued operators defined on dense subspaces of Hilbert spaces $\mathcal{H}$ and $\mathcal{V}$, respectively, and $A$ is the generator of a strongly continuous semigroup $S(t),\ t>0.$ For this we will follow J. M. Coron \cite[Chapter IV]{CoronBookcontrol}.

Let us begin by specifying the definition of controllability.
\begin{definition}
   The system \eqref{Maingeneral} is controllable in time $T>0$ if, for every $u_0,u_T\in D(A)$, there exists an input (or control) map $v:[0,T]\rightarrow D(B)$ such that the solution $u$ of the Cauchy problem
   \begin{equation*}
       \begin{cases}\displaystyle\frac{d u}{dt}=Au+Bv,& 
    \\
\\u(0)=u_{0}, \end{cases}
   \end{equation*}
   reaches $u_T$ at time $T$, that is, $u(T)=u_T.$
\end{definition}
There is an extensive bibliography dedicated to the study of the controllability of the system \eqref{Maingeneral}. We refer, for instance, to \cite{Russell} and \cite{Zuazuasurvey}. A large number of analytic and numerical methods have been developed and used in different contexts, for example, R. Kálmán proposed and proved a criterion to determine whether a finite-dimensional linear system is controllable. This criterion is presented in the following theorem.
\begin{theorem}[Kalman's criterion \cite{KalOptimal}]\label{Kalmantheorem} If $\mathcal{H}$ and $\mathcal{V}$ have finite dimensions $n$ and $m$ respectively, then the system \eqref{Maingeneral} is controllable in time $T>0$ if and only if \begin{equation}\label{kalmancondition}\textnormal{rank}\left[B,AB,\cdots,A^{n-1}B\right]=n.\end{equation}
\end{theorem}
Equality \eqref{kalmancondition} is called the {\it rank Kalman condition} and we can observe that it does not depend on the time $T$ so, in particular, the Kalman's criterion implies that in finite dimension, the system is controllable in any time if it is controllable in some time $T>0$.

Another useful tool to determine whether or not a system is controllable in a time $T>0$ is given by the next theorem. For a proof of this theorem, see, for instance \cite[p. 57]{CoronBookcontrol}.
\begin{theorem}[Observality criterion]\label{observabletheorem}
    The system \eqref{Maingeneral} is controllable at a  time $T>0$ if and only if there exists a constant $c_T>0$ such that
    \begin{equation}\label{generalobservableinequality}
    \displaystyle\int\limits_0^T||B^*S(t)^*z||^2_{\mathcal{V}}dt\geq c_T^2||z||^2_{\mathcal{H}},\ \forall z\in D(A^*),
    \end{equation}
    where $D(A^*)$ denotes the domain of the adjoint operator $A^*$ of $A.$
\end{theorem}
Inequality \eqref{generalobservableinequality} is  usually called the {\it observability inequality} for the system \eqref{Maingeneral} and we will use it later in the proof of our main theorem.
\begin{remark} Let $c({T})>0$ be  the supremum of the constants $c_T>0$ satisfying  the observability inequality \eqref{generalobservableinequality}.
    By following the standard terminology of the control theory, the constant $$ \boxed{\mathscr{C}_{T}:={1}/{c({T})}}$$   is called the {\it controllability cost} of the controllable system \eqref{Maingeneral}.
\end{remark}

\subsection{Left-invariant operators on compact Lie groups} In order to record the equivalence between Fourier multipliers and left-invariant operators let us start with some basics about the Fourier analysis of a compact Lie group.

Let $dx$ be the Haar measure on a compact Lie group $G.$  
  The Hilbert space $L^2(G)$ will be endowed with
   the inner product $$ (f,g)=\int\limits_{G}f(x)\overline{g(x)}dx.$$
   According to the Peter-Weyl theorem the spectral decomposition of $L^2(G)$ can be done in terms of the entries of unitary representations on a compact Lie group $G$. To present such a theorem we will give some preliminaries.
\begin{definition}[Unitary representation of a compact Lie group]
    A continuous and unitary representation of  $G$ on $\mathbb{C}^{\ell}$ is any continuous mapping $\xi\in\textnormal{Hom}(G,\textnormal{U}(\ell)) ,$ where $\textnormal{U}(\ell)$ is the Lie group of unitary matrices of order $\ell\times \ell.$ The integer number $\ell=\dim_{\xi}$ is called the dimension of the representation $\xi$ since it is the dimension of the representation space $\mathbb{C}^{\ell}.$
\end{definition}
\begin{remark}[Irreducible representations] A subspace $W\subset \mathbb{C}^{d_\xi}$ is called $\xi$-invariant if for any $x\in G,$ $\xi(x)(W)\subset W,$ where $\xi(x)(W):=\{\xi(x)v:v\in W\}.$ The representation $\xi$ is irreducible if its only invariant subspaces are $W=\emptyset$ and $W=\mathbb{C}^{d_\xi},$ the trivial ones. On the other hand, 
any unitary representation $\xi$ is a direct sum of unitary irreducible representations. We denote it by $\xi=\xi_1\oplus \cdots\oplus \xi_j,$ with $\xi_i$ being irreducible representations on factors $\mathbb{C}^{d_{\xi_i}}$ that decompose the representation space $$ \mathbb{C}^{d_{\xi}}=\mathbb{C}^{d_{\xi_1}}\oplus \cdots \oplus\mathbb{C}^{d_{\xi_j}}.$$ 
\end{remark}
\begin{definition}[Equivalent representations]
    Two unitary representations $$ \xi\in \textnormal{Hom}(G,\textnormal{U}(d_\xi)) \textnormal{ and  }\eta\in \textnormal{Hom}(G,\textnormal{U}(d_\eta))$$  are equivalent if there exists a linear invertible map $S:\mathbb{C}^{d_\xi}\rightarrow \mathbb{C}^{d_\eta}$ such that for any $x\in G,$ $S\xi(x)=\eta(x)S.$ The mapping $S$ is called an intertwining operator between $\xi$ and $\eta.$ The set of all the intertwining operators between $\xi$ and $\eta$ is denoted by $\textnormal{Hom}(\xi,\eta).$
\end{definition}
\begin{remark}[Schur Lemma]
    In view of the 1905's Schur lemma, if $\xi\in \textnormal{Hom}(G,\textnormal{U}(d_\xi)) $ is irreducible, then $\textnormal{Hom}(\xi,\xi)=\mathbb{C}I_{d_\xi}$ is formed by scalar multiples of the identity matrix  $I_{d_\xi}$ of order $d_\xi.$

\end{remark}
\begin{definition}[The unitary dual]
    The relation $\sim$ on the set of unitary representations $\textnormal{Rep}(G)$ defined by: {\it $\xi\sim \eta$ if and only if $\xi$ and $\eta$ are equivalent representations,} is an equivalence relation. The quotient 
$$
    \widehat{G}:={\textnormal{Rep}(G)}/{\sim}
$$is called the unitary dual of $G.$
\end{definition}
The unitary dual encodes all the Fourier analysis on the group. The Fourier transform is defined as follows.
\begin{definition}[Group Fourier transform]
    If $\xi\in \textnormal{Rep}(G),$ the Fourier transform $\mathscr{F}_{G}$ associates to any $f\in C^\infty(G)$ a matrix-valued function $\mathscr{F}_{G}f$ defined on $\textnormal{Rep}(G)$ as follows
$$ (\mathscr{F}_{G}f)(\xi) \equiv   \widehat{f}(\xi)=\int\limits_Gf(x)\xi(x)^{*}dx,\,\,\xi\in \textnormal{Rep}(G). $$ 
\end{definition}
\begin{remark}[The Fourier inversion formula on a compact Lie group]
The discrete Schwartz space $\mathscr{S}(\widehat{G}):=\mathscr{F}_{G}(C^\infty(G))$ is the image of the Fourier transform on the class of smooth functions. This operator admits a unitary extension from $L^2(G)$ into $\ell^2(\widehat{G}),$ with 
\begin{equation}
 \ell^2(\widehat{G})=\left\{\phi:\forall [\xi]\in \widehat{G},\,\phi(\xi)\in \mathbb{C}^{d_\xi\times d_\xi}\textnormal{ and }\Vert \phi\Vert_{\ell^2(\widehat{G})}:=\left(\sum_{[\xi]\in \widehat{G}}d_{\xi}\Vert\phi(\xi)\Vert_{\textnormal{HS}}^2\right)^{\frac{1}{2}}<\infty \right\}.   
\end{equation} The norm $\Vert\phi(\xi)\Vert_{\textnormal{HS}}$ is the standard Hilbert-Schmidt norm of matrices. The Fourier inversion formula takes the form
\begin{equation}
    f(x)=\sum_{[\xi]\in \widehat{G}} d_{\xi}\textnormal{Tr}[\xi(x)\widehat{f}(\xi)],\,f\in L^2(G),
\end{equation}where the summation is understood in the sense that from any equivalence class $[\xi]$ we choose randomly a unitary representation.      
\end{remark}
 \begin{remark}
The Plancherel theorem for the group Fourier transform takes the form
\begin{equation}\label{Plancherel}
         \forall f\in L^2(G),\,\,\, \Vert f\Vert_{L^2(G)}=\left(\sum_{[\xi]\in \widehat{G}} d_{\xi}\Vert\widehat{f}(\xi)\Vert_{\textnormal{HS}}^2\right)^{\frac{1}{2}}.
\end{equation}
\end{remark}

 Let  $A:C^\infty(G)\rightarrow C^\infty(G)$ be a continuous linear operator with respect to the standard Fr\'echet structure on $C^\infty(G).$ There is a way of associating to the operator $A$ a matrix-valued function $\sigma_A$ defined on the non-commutative phase space $G\times \widehat{G}$ to rewrite the operator $A$ in terms of the Fourier inversion formula and in terms of the Fourier transform. Such an expression is called the dequantisation formula. To introduce it we require the following definition.
\begin{definition}[Right convolution kernel of an operator]
 The Schwartz kernel theorem associates to $A$ a kernel $K_A\in \mathscr{D}'(G\times G)$ such that
$$   Af(x)=\int\limits_{G}K_{A}(x,y)f(y)dy,\,\,f\in C^\infty(G).$$ The distribution defined via $R_{A}(x,y):=K_A(x,xy^{-1})$ that provides the convolution identity
$$   Af(x)=\int\limits_{G}R_{A}(x,y^{-1}x)f(y)dy,\,\,f\in C^\infty(G),$$
is called the right-convolution kernel of $A.$  
\end{definition}

\begin{remark}[The dequantisation formula]
 Now, we will associate  a global symbol $\sigma_A:G\times \textnormal{Rep}(G)\rightarrow \cup_{\ell\in \mathbb{N}}\mathbb{C}^{\ell\times \ell}$ to $A.$ Indeed, for a given $x_0\in G$, we can consider the continuous linear operator $A_{x_0}:C^{\infty}(G)\rightarrow C^{\infty}(G)$ defined by $$A_{x_0}f(x)=\int\limits_GR_A(x_0,y^{-1}x)f(y)dy=(f\ast R_A(x_0,\cdot))(x),$$ and after taking the Fourier transform we get 
 $$ \widehat{A_{x_0}f}(\xi)= \widehat{R_{A}(x_0,\cdot)}(\xi)\widehat{f}(\xi). $$ Then, the Fourier inversion formula gives the following representation of the operator $A_{x_0}$ in terms of the Fourier transform,
 \begin{equation}\label{Quantisation:formula}
     A_{x_0}f(x)=\sum_{[\xi]\in \widehat{G}}d_\xi\textnormal{Tr}[\xi(x)\widehat{R_{A}(x_0,\cdot)}(\xi)\widehat{f}(\xi)],\,f\in C^\infty(G),
 \end{equation}
 and, therefore,
  \begin{equation}\label{Quantisation:formula1}
     Af(x)=A_{x}f(x)=\sum_{[\xi]\in \widehat{G}}d_\xi\textnormal{Tr}[\xi(x)\widehat{R_{A}(x,\cdot)}(\xi)\widehat{f}(\xi)],\,f\in C^\infty(G).
 \end{equation}
 We define the symbol of $A$ at $(x,\xi)\in G\times \textnormal{Rep}(G)$ as
 \begin{equation}
    \sigma_A(x,\xi):=\widehat{R_{A}(x,\cdot)}(\xi),
 \end{equation}so that
\begin{equation}\label{Quantisation:formula2}
     Af(x)=\sum_{[\xi]\in \widehat{G}}d_\xi\textnormal{Tr}[\xi(x)\sigma_{A}(x,\xi)\widehat{f}(\xi)],\,f\in C^\infty(G).
 \end{equation}The formula \eqref{Quantisation:formula2} is independent of the choice of the representation $\xi$ from any equivalent class $[\xi]\in \widehat{G}.$ This is a consequence of the Fourier inversion formula.
\end{remark}
  In the following quantisation theorem, we observe that the distribution $\sigma_A$ in \eqref{Quantisation:formula2} is unique and can be written in terms of the operator $A,$ see Theorems  10.4.4 and 10.4.6 of \cite[Pages 552-553]{Ruz}.
\begin{theorem}\label{The:quantisation:thm}
    Let $A:C^\infty(G)\rightarrow C^\infty(G) $ be a continuous linear operator. The following statements are equivalent.
    \begin{itemize}
        \item The distribution $\sigma_A:G\times \widehat{G}\rightarrow \cup_{\ell\in \mathbb{N}}\mathbb{C}^{\ell\times \ell}$ satisfies the quantisation formula
        \begin{equation}\label{Quantisation:3}
            \forall f\in C^\infty(G),\,\forall x\in G,\,\, Af(x)=\sum_{[\xi]\in \widehat{G}}d_\xi\textnormal{Tr}[\xi(x)\sigma_{A}(x,\xi)\widehat{f}(\xi)].
        \end{equation}
        \item $ 
            \forall (x,\xi)\in G\times \textnormal{Rep}(G),\, \sigma_{A}(x,\xi)=\widehat{R_{A}(x,\cdot)}(\xi).
        $ \\
        \item $ 
          \forall (x,\xi),\, \sigma_A(x,\xi)=\xi(x)^{*}A\xi(x),$ where $ A\xi(x):=(A\xi_{ij}(x))_{i,j=1}^{d_\xi}.  
        $ 
    \end{itemize}
\end{theorem}
%\begin{remark}
 %   In view of the quantisations formulae \eqref{Quantisation:formula2} and \eqref{Quantisation:3}, a symbol $\sigma_A$ can be considered as a mapping defined on $G\times \widehat{G}$ or as a mapping  defined on $$G\times \textnormal{Rep}(G)$$ by identifying all the values $\sigma_A(x,\xi)=\sigma_A(x,\xi')=\sigma(x,[\xi])$ when $\xi',\xi\in [\xi].$
%\end{remark}
\begin{example}[Spectrum of the Laplacian] Let $$\mathbb{X}=\{X_1,\cdots,X_n\}$$ be an orthonormal basis of the Lie algebra $\mathfrak{g}.$ The positive Laplacian on $G$ is the second order differential operator 
\begin{equation}
    \mathcal{L}_G=-\sum_{j=1}^nX_j^2.
\end{equation}The operator $ \mathcal{L}_G$ is independent of the choice of the orthonormal basis $\mathbb{X}$ of $\mathfrak{g}.$ The $L^2$-spectrum of $\mathcal{L}_G$ is a discrete set that can be enumerated in terms of the unitary dual $\widehat{G}$
\begin{equation}
    \textnormal{Spect}(\mathcal{L}_G)=\{\lambda_{[\xi]}:[\xi]\in \widehat{G}\}.
\end{equation} Of particular interest for our further analysis will be the Japanese bracket function
\begin{equation}
    \langle t\rangle:=(1+t)^{\frac{1}{2}},\,t\geq -1.
\end{equation}In particular the symbol of the operator $ \langle \mathcal{L}_G\rangle=(1+\mathcal{L}_G)^{\frac{1}{2}}$ is given by
\begin{equation}\label{Japanne:bracket:G}
     \sigma_{\langle \mathcal{L}_G\rangle}([\xi]):= \langle \xi \rangle I_{d_\xi}, \,\,\,\langle \xi \rangle:=\langle \lambda_{[\xi]} \rangle.  
\end{equation}

\end{example}

Consider the action of the group $G$ on $C^{\infty}(G)$ given by $\rho:(x,f)\longmapsto f\circ L_x,$ where $L_x(y):=xy.$ A continuous operator $A:C^{\infty}(G)\rightarrow C^{\infty}(G)$ is called {\it left-invariant} if $A$ commutes with $\rho(x,\cdot)$ for all $x\in G,$ i.e., if it satisfies the following property
$$A(f\circ L_x)=(Af)\circ L_x,\ f\in C^{\infty}(G),\ x\in G.$$ 
\begin{proposition}[\cite{Ruz}] The following statements are equivalent:
\begin{itemize}
    \item[$i)$] $A$ is left-invariant.
    \item[$ii)$] $R_A(x,y)=R_A(zx,y),\ \forall x,y,z\in G.$
    \item[$iii)$] $\sigma_A(x_1,\xi)=\sigma_A(x_2,\xi),$ for all $x_1,x_2\in G$ and $\xi\in \textnormal{Rep}(G).$
    \item[$iv)$] $A_{x_0}=A,\ \forall x_0\in G.$ 
\end{itemize}
\end{proposition}
%\begin{proof} Let $x,z\in %G.$ Then
%    \begin{align*}
%        \displaystyle \sigma_A(x,\xi)=\sigma_A(zx,\xi),\ \forall\xi\ &\Leftrightarrow\ \widehat{R_{A}(x,\cdot)}(\xi)=\widehat{R_{A}(zx,\cdot)}(\xi),\ \forall\xi\\
%        \\
%        &\displaystyle\Leftrightarrow\ \int\limits_GR_{A}(x,y)\xi(y)^{*}dy=\int\limits_GR_{A}(zx,y)\xi(y)^{*}dy,\ \forall\xi\\
%        \\
%        &\displaystyle\Leftrightarrow\ R_{A}(x,y)=R_{A}(zx,y),\ \forall y\\
%        \\
%        &\Leftrightarrow\ \int\limits_GR_{A}(x,y^{-1}zx)f(y)dy=\int\limits_GR_{A}(zx,y^{-1}zx)f(y)dy,\ \forall f\\
%        \\
%        &\Leftrightarrow\ \int\limits_GR_{A}(x,y^{-1}x)f(zy)dy=\int\limits_GR_{A}(zx,y^{-1}zx)f(y)dy,\ \forall f\\
%        \\
%        &\Leftrightarrow\ A(f\circ L_z)(x)=[(Af)\circ L_z](x).
%    \end{align*}
%By taking $x=x_1$ and $z=x_2x_1^{-1},$ we obtain the equivalence of $i),$ $ii)$ and $iii)$. To prove $i)\Longleftrightarrow iv)$, it is sufficient to observe that \begin{equation*}
%    (A_{x_0}f)(x)=A(f\circ L_{xx_0^{-1}})(x_0)\ \textnormal{and}\ Af(x)=[(Af)\circ L_{xx_0^{-1}}](x_0).
%\end{equation*}The proof is complete.
%\end{proof}
In particular, the proposition above says that for a left-invariant operator $A$, the symbol $\sigma_A(x,\xi)$ does not depend on $x$, so in this case we can define $\sigma_A(\xi):=\sigma_A(x,\xi)$ for any $x\in G.$

\begin{remark} Let $A,B:C^\infty(G)\rightarrow C^\infty(G)$ be  continuous linear operators. Assume that $A$ is the generator of a $C_0$-semigroup $S(t),$ that is
\begin{equation}
    \forall f\in C^{\infty}(G),\,\,Af=\lim_{t\rightarrow 0}\frac{1}{t}\left(S(t)f-f\right),
\end{equation}where the limit is taken with respect to the $L^2$-norm on $G$. If $A$ is left-invariant, note that
    \begin{equation}\label{symbolproperties}
        \forall \xi \in \widehat{G},\,\,\,\sigma_{S(t)}(\xi)=e^{t\sigma_A(\xi)} \textnormal{ and }\sigma_{B^\ast}(\xi)=\sigma_{B}(\xi)^*.
    \end{equation}
\end{remark}

\begin{remark}[Fourier multipliers on compact manifolds vs invariant operators] \label{main:remark:global}
If $A:C^\infty(G)\rightarrow L^2(G)$ is an invariant operator (with respect to the Laplacian $\mathcal{L}_G$) then we have two notions of global symbols for $A.$ One is defined in terms of the representation theory of the group $G$ and we will denote this symbol by $(\sigma_{A}(\xi))_{[\xi]\in \widehat{G}},$ and the other one is that defined when we consider the compact Lie group as a manifold, and in this case the symbol will be denoted by $(\sigma_{A}(l))_{l\in \mathbb{N}_0}.$ The relation of this two symbols has been established in \cite[Page 25]{DelRuzTrace1}. Now, we describe this relation. In the setting of compact Lie groups the unitary
dual being discrete, we can enumerate the unitary dual as $[\xi_{j}],$ for $j\in\mathbb{N}_0.$  In this way we fix the orthonormal basis
\begin{equation}
\{e_{jk}\}_{k=1}^{d_j}=\{ d_{\xi_{j}}^{\frac{1}{2}}(\xi_{j})_{il}  \}_{i,l=1}^{d_{\xi_j}}
\end{equation}
where $d_{j}=d_{\xi_j}^2.$ Then, we have the subspaces $H_{j}=\textnormal{span}\{(\xi_j)_{i,l}:i,l=1,\cdots,d_{\xi_j}\}.$ With the notation above we have
$$\sigma_{A}(l)=
\begin{bmatrix}
    \sigma_{A}({\xi_{l}}) & 0_{d_{\xi_{l}}\times d_{\xi_{l}} } & 0_{d_{\xi_{l}}\times     d_{\xi_{l}} } & \dots  & 0_{d_{\xi_{l}}\times d_{\xi_{l}} } \\
     0_{d_{\xi_{l}}\times d_{\xi_{l}} } & \sigma_{A}({\xi_{l}})& 0_{d_{\xi_{l}}\times d_{\xi_{l}} } & \dots  & 0_{d_{\xi_{l}}\times d_{\xi_{l}} } \\
    \vdots & \vdots & \vdots & \ddots & \vdots \\
   0_{d_{\xi_{l}}\times d_{\xi_{l}} } & 0_{d_{\xi_{l}}\times d_{\xi_{l}} } &0_{d_{\xi_{l}}\times d_{\xi_{l}} } & \dots  & \sigma_{A}({\xi_{l}})
\end{bmatrix}_{d_{l}\times d_{l}}.$$
\end{remark}

\section{Control theory on Hilbert spaces: symbol criteria}\label{Control:Compact}
In this section we present a controllability criterion for the Cauchy problem associated to Fourier multipliers on Hilbert spaces. As it was shown in the previous section these are operators leaving invariant a fixed decomposition of the Hilbert space in subspaces of finite dimension.   Such a result is presented below as Theorem \ref{Main:Theorem:HS} and   will be formulated   in terms of the global matrix-valued symbols of the operators. For our further analysis we require the following definition. 
\begin{definition}[Image of the Cauchy problem under the Fourier transform]\label{Fourier:costs}
   Let $\mathcal{H}$ be a complex Hilbert space and let $\mathcal{H}^{\infty}\subset \mathcal{H}$ be a dense
linear subspace of $\mathcal{H}$. Let $\mathcal{H}=\bigoplus_{j}H_j$ be a decompositon of $\mathcal{H}$ in orthogonal subspaces $H_j$ of dimension $d_{j}\in\mathbb{N}.$ Let $A,B:\mathcal{H}^\infty\rightarrow \mathcal{H}$ be Fourier multipliers relative to the decomposition $\{H_j\}_{j\in \mathbb{N}_0}.$ Consider the Cauchy problem
\begin{equation}\label{Cauchy:H}
\textnormal{(CP): }\begin{cases}\displaystyle\frac{d u}{dt}=Au+Bv,& 
    \\
\\u(0)=u_{0}\in \mathcal{H}^\infty.
\text{ } \end{cases} 
\end{equation}We define the image of $\textnormal{(CP)}$ under the  Fourier transform relative the the decomposition $\{H_j\}_{j\in \mathbb{N}_0}$ to be the infinite family of finite-dimensional dynamical systems
\begin{equation}\label{Cauchy:ell}
(\textnormal{(CP)},{\ell}):\begin{cases}\displaystyle\frac{d\widehat{ u}(\ell)}{dt}=\sigma_A(\ell)\widehat{u}(\ell)+\sigma_B(\ell)\widehat{v}(\ell),& 
    \\
\\\widehat{u(0)}(\ell)=\widehat{u}_{0}(\ell)\in \mathbb{C}^{d_\ell}.
\text{ } \end{cases},\,\ell \in \mathbb{N}_0.
\end{equation}
\end{definition}
The previous definition is motivated by the reduction of the controllability of the system \eqref{Cauchy:H}
\begin{lemma}\label{Lemma:ell}The following statements are equivalent.
\begin{itemize}
\item[(1)] For all $\ell\in\mathbb{N},$ the Cauchy problem $(\textnormal{(CP)},{\ell})$ in  \eqref{Cauchy:xi} is a controllable dynamical system at a time $T>0$.
    \item[(2)] $\forall\ell\in\mathbb{N}_0$, the Kalman condition 
    \begin{equation}
        \textnormal{rank}\left[\sigma_B(\ell),\ \sigma_A(\ell)\sigma_B(\ell),\ \cdots,\ \sigma_A(\ell)^{{d_\ell}-1}\sigma_B(\ell)\right]=d_{\ell}
    \end{equation}is satisfied.
    \item[(3)]  $\forall \ell\in\mathbb{N}_0,\ \exists c= c(\ell,T)>0$ such that \begin{equation}\label{Observability}
\displaystyle\int\limits_0^T||\sigma_B(\ell)^*\exp{(t\sigma_A(\ell)^*)}z||^2_{\mathbb{C}^{d_\ell}}dt\geq c(\ell,T)^2||z||^2_{\mathbb{C}^{d_\ell}},\ \forall z\in \mathbb{C}^{\ell}.
    \end{equation}
\end{itemize}
\end{lemma}
\begin{proof}Note that the equivalence $\textnormal{(1)}\Longleftrightarrow \textnormal{(2)}$ follows from Kalman's criterion (see Theorem \ref{Kalmantheorem}). On the other hand, the equivalence  $\textnormal{(1)}\Longleftrightarrow \textnormal{(3)}$ is nothing else that the observability criterion in Theorem \ref{observabletheorem}. It is clear then the equivalence  $\textnormal{(2)}\Longleftrightarrow \textnormal{(3)}.$ The proof of Lemma \ref{Lemma:xi} is complete.    
\end{proof}
\begin{remark} Let $c_{\ell,T}$ be  the supremum of the constants $c= c(\ell,T)>0$ satisfying  the observability inequality \eqref{Observability}.
    According to the usual nomenclature of the control theory, the constant $$ \boxed{\mathscr{C}_{\ell,T}:={1}/{c_{\ell,T}}}$$   is called the controllability cost of the Cauchy problem  \eqref{Cauchy:ell}.
\end{remark}

\begin{definition}\label{Definition:global:consts}
    We will say that the image of the Cauchy problem \eqref{Cauchy:H} under the Fourier transform associated to the decomposition $\{H_j\}_{j\in \mathbb{N}},$ has a {\it finite global controllability cost} if 
    \begin{equation}
     \mathscr{C}_{T}:=   \sup_{\ell\in\mathbb{N} }\mathscr{C}_{\ell,T}<\infty.
    \end{equation}
\end{definition}

Now, we present the following criterion for the controllability of the Cauchy problem for Fourier multipliers on Hilbert spaces.

\begin{theorem}\label{Main:Theorem:HS}
Let $\mathcal{H}$ be a complex Hilbert space and let $\mathcal{H}^{\infty}\subset \mathcal{H}$ be a dense
linear subspace of $\mathcal{H}$. Let $\mathcal{H}=\bigoplus_{j}H_j$ be a decompositon of $\mathcal{H}$ in orthogonal subspaces $H_j$ of dimension $d_{j}\in\mathbb{N}.$ Let $A,B:\mathcal{H}^\infty\rightarrow \mathcal{H}$ be Fourier multipliers relative to the decomposition $\{H_j\}_{j\in \mathbb{N}_0}.$    
\begin{itemize}
        \item[(1)] If the Cauchy problem
    \begin{equation}\label{MainII} \begin{cases}\displaystyle\frac{d u}{dt}=Au+Bv,& 
    \\
\\u(0)=u_{0}\in \mathcal{H}^\infty,
\text{ } \end{cases}
\end{equation}
    is controllable, then for any $\ell\in \mathbb{N}_0,$ the global symbols $\sigma_A(\ell)$ and $\sigma_B(\ell)$ of $A$ and $B,$ respectively, satisfy the Kalman condition:
    \begin{equation}\label{Rank:Hilbert:spaces}
     \textnormal{rank}\left[\sigma_B(\ell),\ \sigma_A(\ell)\sigma_B(\ell),\ \cdots,\ \sigma_A(\ell)^{{d_\ell}-1}\sigma_B(\ell)\right]=d_{\ell}.   
    \end{equation} Additionally, if $A$ generates a strongly continuous semigroup on $\mathcal{H},$ the image of the Cauchy problem \eqref{MainII} under the Fourier transform relative to the decomposition $(H_j)_{j\in \mathbb{N}},$  has a {\it finite global controllability cost} at time $T>0,$ that is 
    $$  \mathscr{C}_{T}:=   \sup_{\ell\in \mathbb{N}_0 }\mathscr{C}_{\ell,T}<\infty.$$
    Moreover, 
    $$   \mathscr{C}_{T}\leq \Tilde{\mathscr{C}}_T, $$ where  $\tilde{\mathscr{C}_{T}}$ is  the controllability cost of \eqref{MainII}.
    \item[(2)] Conversely, assume that $A$ is the generator of a strongly continuous semigroup on $\mathcal{H},$ and that the Kalman condition \eqref{Rank:Hilbert:spaces} is satisfied for each $\ell\in\mathbb{N}_0$. Assume that the image of the Cauchy problem \eqref{MainII} under the Fourier transform relative to the decomposition $(H_j)_{j\in \mathbb{N}},$  has a {\it finite global controllability cost} in time $T>0,$ that is, $$  \mathscr{C}_{T}:=   \sup_{\ell\in \mathbb{N}_0 }\mathscr{C}_{\ell,T}<\infty. $$ Then, the Cauchy problem  \eqref{MainII} is controllable at time $T>0,$ and its controllability costs $\tilde{\mathscr{C}_{T}}$ satisfies the inequality
    \begin{equation}
      \mathscr{C}_{T}\geq   \tilde{\mathscr{C}_{T}}.
    \end{equation}
    \end{itemize}
\end{theorem}
\begin{proof}For the proof of $\textnormal{(1)}$ let us analyse the image of \eqref{MainII} under the Fourier transform associated to the decomposition $H_j,$ $j\in \mathbb{N},$ in order to deduce \eqref{Rank:Hilbert:spaces}. For the proof of $\textnormal{(2)},$ by following the standard strategy of the control theory, we will reduce the controllability of the system \eqref{MainII}  to the validity of the observability inequality \eqref{generalobservableinequality} in Theorem \ref{observabletheorem}.
\begin{itemize}
    \item Assume that the Cauchy problem \eqref{MainII} is controllable. By fixing $\ell\in\mathbb{N}_0,$ and taking the Fourier transform relative to the decomposition $\{H_j\}_{j\in \mathbb{N}_0},$  we get 
    \begin{equation*}
    \displaystyle\frac{d\widehat{u}(\ell)}{dt}+\widehat{Au}(\ell)=\widehat{Bv}(\ell)\end{equation*}
    and, since $A,B$  are Fourier multipliers we have the following identity in terms of the symbols $\sigma_A$ and $\sigma_B$ of $A$ and $B,$ respectively,
    \begin{equation}\label{MaintransformedM}
    \displaystyle\frac{d\widehat{u}(\ell)}{dt}+\sigma_A(\ell)\widehat{u}(\ell)=\sigma_B(\ell)\widehat{v}(\ell).
    \end{equation}
    This is a dynamical system in the set of square matrices of order $d_\ell.$

     In order to prove the controllability of \eqref{MaintransformedM} let us take $$\zeta_0,\zeta_T\in\mathbb{C}^{d_\ell}.$$ For any $\ell'\in \mathbb{N}_0$ define $$ u_0=(e_{\ell'},\zeta_0)\delta_{\ell,\ell'}\textnormal{  and  }u_T:=(e_\ell,\zeta_T)\delta_{\ell,\ell'}.$$  Observe that the Fourier coefficients of $u_0$ and of $u_T$ satisfty that
     $$  \forall \ell\neq \ell', \,\,\,\widehat{u}_0(\ell')=0_{\mathbb{C}^{d_{\ell'}}}=\widehat{u}_T(\ell').$$
     Since each $H_{j}\subset \mathcal{H}^\infty, $ the vectors
      $u_0,u_T$ belong to $\mathcal{H}^\infty$ and they satisfy that $\widehat{u_0}(\ell)=\zeta_0$ and $\widehat{u_T}(\ell)=\zeta_T.$ Since \eqref{MainII} is controllable, there exists an input function $v$ such that the solution $u$ of \eqref{MainII} satisfies $u(T)=u_T,$ so $\widehat{u}(\ell)$ is a solution of \eqref{MaintransformedM} with $\widehat{u}(\ell)(0)=\zeta_0$ and $\widehat{u}(\ell)(T)=\zeta_T$, i.e., \eqref{MaintransformedM} is controllable.     
    By Kalman's criterion (see Theorem \ref{Kalmantheorem}) we conclude that
    \begin{equation}\label{Kalman:proof:M}
        \textnormal{rank}\left[\sigma_B(\ell),\ \sigma_A(\ell)\sigma_B(\ell),\ \cdots,\ \sigma_A(\ell)^{{d_\ell}-1}\sigma_B(\ell)\right]=d_{\ell}.
    \end{equation}   To end the proof of (1) we have to prove the estimate of the {\it globally finite controllability cost} of the system \eqref{MaintransformedM}. In view of Theorem \ref{observabletheorem}, we have the observability inequality
    \begin{equation}\label{observableinequality:u}
    \displaystyle\int\limits_0^T||B^*S(t)^*f||^2_{\mathcal{H}}dt\geq \left(\frac{1}{\Tilde{\mathscr{C}}_T}\right)^2||f||^2_{\mathcal{H}},\ \forall f\in \mathcal{H}^\infty,
    \end{equation}
    where $S(t)$ is the $C_0$-semigroup generated by $A$.  Via Plancherel theorem it is equivalent to the inequality
     \begin{align}\label{Plancherel:I:th}
    \int\limits_0^T\sum\limits_{\ell\in \mathbb{N}_0}||\sigma_B(\ell)^*\exp{(t\sigma_A(\ell)^*)}\widehat{f}(\ell)||^2_{ \mathbb{C}^{d_\ell} }dt\geq \left(\frac{1}{\Tilde{\mathscr{C}}_T}\right)^2\sum\limits_{\ell\in \mathbb{N}_0 }||\widehat{f}(\ell)||^2_{ \mathbb{C}^{d_\ell} }\,\,\forall f\in \mathcal{H}^\infty.
    \end{align} Now, if $\ell_0$ is fixed, and $z\in \mathbb{C}^{d_{\ell_0}}\setminus\{0\}$ is an arbitrary coordinate vector, let us consider the vector $v_z\in \mathcal{H}^\infty$ determined by the following Fourier coefficients
    \begin{equation}\label{test}
       \widehat{v}_z(\ell)=z\delta_{\ell,\ell_0},\,\,\ell\in \mathbb{N}_0. 
    \end{equation}Plugging \eqref{test} into \eqref{Plancherel:I:th} we have that
$$ \int\limits_0^T||\sigma_B(\ell_0)^*\exp{(t\sigma_A(\ell_0)^*)}z||^2_{ \mathbb{C}^{d_{\ell_0}} }dt\geq \left(\frac{1}{\Tilde{\mathscr{C}}_T}\right)^2||z||^2_{ \mathbb{C}^{d_\ell} }. $$ This inequality is the observability inequality of the system \eqref{MaintransformedM} when $\ell=\ell_0.$ Note that if $\mathscr{C}_{\ell,T}$ is the controllability costs of \eqref{MaintransformedM} when $\ell=\ell_0$, then we have the inequality  
$$ \left(\frac{1}{ \mathscr{C}_{\ell_0,T}}\right)^2\geq  \left(\frac{1}{\Tilde{\mathscr{C}}_T}\right)^2$$ from which we deduce that
 $$  \mathscr{C}_{T}=\sup_{\ell_0} \mathscr{C}_{T,\ell_0}\leq \Tilde{\mathscr{C}}_T, $$ as desired. The proof of (1) is complete.
 
    \item  Now, let us prove (2). So, 
    conversely, suppose that $$\forall\ell\in\mathbb{N}_0,\ \textnormal{rank}\left[\sigma_B(\ell),\ \sigma_A(\ell)\sigma_B(\ell),\ \cdots,\ \sigma_A(\ell)^{{d_\ell}-1}\sigma_B(\ell)\right]=d_\ell.$$ 
  We want to prove the controllability of the Cauchy problem \eqref{MainII}  in any time $T>0$. According to Theorem \ref{observabletheorem}, it is sufficient to show that there exists $c_T>0$ such that 
    \begin{equation}\label{observable:inequality:Hilbert}
    \displaystyle\int\limits_0^T||B^*S(t)^*f||^2_{\mathcal{H}}dt\geq c_T^2||f||^2_{\mathcal{H}},\ \forall f\in \mathcal{H}^\infty,
    \end{equation}
    where $S(t)$ is the $C_0$-semigroup generated by $A$. 
    By the Kalman criterion, we know that the system
    \begin{equation}
        \displaystyle\frac{d\gamma_{\ell}}{dt}+\sigma_A(\ell)\gamma_{\ell}=\sigma_B(\ell)v_{\ell}
    \end{equation}
     is controllable for every $\ell\in\mathbb{N}_0.$ In consequence, the inequality 
    \begin{equation*}
    \displaystyle\int\limits_0^T||\sigma_B(\ell)^*\exp{(t\sigma_A(\ell)^*)}z||^2_{\mathbb{C}^{d_\ell}}dt\geq c_{\ell,T}^2||z||^2_{ \mathbb{C}^{d_\ell}  },\ \forall z\in \mathbb{C}^{d_\ell},
    \end{equation*}holds and let us denote by $c_{\ell,T}>0$ the largest constant that satisfies this inequality.
    In particular, for $z=\widehat{f}(\ell)$  we get
$$ \int\limits_0^T||\sigma_B(\ell)^*\exp{(t\sigma_A(\ell)^*)}\widehat{f}(\ell)||^2_{ \mathbb{C}^{d_\ell} }dt\geq c_{\ell,T}^2||\widehat{f}(\ell)||^2_{ \mathbb{C}^{d_\ell}  }.  $$
By  summing over $\ell\in \mathbb{N}_0,$  we obtain
$$ \sum\limits_{ \ell\in \mathbb{N}_0 }\int\limits_0^T||\sigma_B(\ell)^*\exp{(t\sigma_A(\ell)^*)}\widehat{f}(\ell)||^2_{ \mathbb{C}^{d_\ell} }dt\geq \sum\limits_{ \ell\in \mathbb{N}_0 } c_{\ell,T}^2||\widehat{f}(\ell)||^2_{\mathbb{C}^{d_\ell}}.  $$
Consequently,
    \begin{align*}
    \int\limits_0^T\sum\limits_{\ell\in \mathbb{N}_0}||\sigma_B(\ell)^*\exp{(t\sigma_A(\ell)^*)}\widehat{f}(\ell)||^2_{ \mathbb{C}^{d_\ell} }dt\geq \sum\limits_{\ell\in \mathbb{N}_0 }c_{\ell,T}^2||\widehat{f}(\ell)||^2_{ \mathbb{C}^{d_\ell} }.
    \end{align*}  
    Using the semigroup property in \eqref{symbolproperties:b} we have that
    $$\sigma_B(\ell)^*\exp{(t\sigma_A(\ell)^*)}\widehat{f}(\ell)=\sigma_{B^*S(t)^*}(\ell)\widehat{f}(\ell)$$  we have that $$\int\limits_0^T\sum\limits_{ \ell\in \mathbb{N}_0 }||\sigma_{B^*S(t)^*}(\ell)\widehat{f}(\ell)||^2_{\mathbb{C}^{d_\ell}}dt\geq c_T^2\sum\limits_{ \ell\in \mathbb{N}_0 }||\widehat{f}(\ell)||^2_{\mathbb{C}^{d_\ell}},$$
    where $$c_T:=\inf\limits_{\ell\in \mathbb{N}_0}c_{\ell,T}.$$
    By Plancherel's formula (see \eqref{Plancherel:H})  we have that \begin{equation*}\displaystyle\int\limits_0^T||B^*S(t)^*f||^2_{\mathcal{H}}dt=\int\limits_0^T\sum\limits_{ \ell\in \mathbb{N}_0 }||\sigma_{B^*S(t)^*}(\ell)\widehat{f}(\ell)||^2_{\mathbb{C}^{d_\ell}}dt\end{equation*}
    and since
    \begin{equation*}
        ||f||^2_{\mathcal{H}}=\sum\limits_{\ell\in \mathbb{N}_0 }||\widehat{f}(\ell)||^2_{\mathbb{C}^{d_\ell} },
    \end{equation*}
    the equality \eqref{observable:inequality:Hilbert} holds with $$c_T:=\inf\limits_{\ell\in \mathbb{N}_0}c_{\ell,T}= \inf\limits_{ \ell\in \mathbb{N}_0 }1/\mathscr{C}_{\ell,T}=  1/ \sup_{\ell\in \mathbb{N}_0 }\mathscr{C}_{\ell,T}<\infty.$$ The proof of (2) is complete. Indeed, note that the
    controllability cost $\tilde{\mathscr{C}_{T}}$ of \eqref{MainII} is the infimum of the constants $c_T>0$ satisfying  \eqref{observable:inequality:Hilbert}, from where we deduce that
    \begin{equation}
      \mathscr{C}_{T}\geq   \tilde{\mathscr{C}_{T}}.
    \end{equation} 
\end{itemize}   
Having proved (1) and (2) the proof of Theorem \ref{Main:Theorem:HS} is complete.
\end{proof}

\section{Applications}\label{Applications}
\subsection{Decoupling Algorithm}
In this section we present a variety of applications of the criterion of controllability in Theorem \ref{MainII} and/or of the following algorithm developed during its proof.
\begin{algo}\label{algor} We start by fixing two densely defined operators $A$ and $B$  on a separable Hilbert space $\mathcal{H}$ satisfying the hypothesis of Theorem \ref{Main:Theorem:HS}.
\begin{itemize}
\item   $\texttt{Algorithm:} $ Criterion for the controllability of the Cauchy problem \eqref{MainII}. 
\item $\texttt{Input:} $ To give the  (global) controllability cost of the  systems defined below in  \eqref{Rep:intro}.
 \item $\texttt{Output:} $ To estimate the cost of controllability of the control system  \eqref{MainII} from above.
\item  $\texttt{Instructions:} $

    \item[] $\texttt{Step 1.} $ To compute the group Fourier transform  of the system \eqref{MainII}. Then one obtains an infinite number of control systems
    \begin{equation}\label{Rep:intro} \begin{cases}\displaystyle\frac{d \widehat{u}(\ell)}{dt}=\sigma_{A}(\ell)\widehat{u}(\ell)+\sigma_B(\ell)\widehat{v}(\xi),& 
    \\
\\\widehat{u(0)}(\ell)=\widehat{u}_{0}(\ell),
\text{ } \end{cases},\,\ell\in \mathbb{N}_0, 
\end{equation} At this point we recognize the input of our algorithm:
    \item[] $\texttt{Step 2.} $ To reduce  the controllability of the system   \eqref{MainII} to an observability inequality that involves the (global) controllability cost of the  systems in  \eqref{Rep:intro}.
    \item [] $\texttt{Step 3.}$ To estimate the controllability cost of the system in \eqref{MainII}  in terms of the  (global) controllability cost of the  systems in  \eqref{Rep:intro}.
    \item[]  $\texttt{Step 4.}$ If the estimated (global) controllability cost of the  systems in  \eqref{Rep:intro} is finite, we are able to deduce  the controllability of  \eqref{Rep:intro}.
\end{itemize}
\end{algo}

\subsection{Control for the Cauchy problem  on compact manifolds}\label{Control:compact:manifolds} Let $M$ be a closed manifold (compact and without boundary). Consider $E\in \Psi^\nu_{cl,+}(M)$ be a classical positive pseudo-differential operator of order $\nu>0$ on $M,$ see H\"ormander \cite{Hormander1985III}.  Let
\begin{itemize}
    \item $H_{j}=\textnormal{Ker}(E-\lambda_jI)$ be the family of eigenspaces of $E$,
    \item $P_{j}:L^2(M)\to H_{j}$ be the corresponding orthogonal projections,
    \item  $\{d_{j}\}_{j\in\mathbb{N}_{0}}\subset\mathbb{N}$ be the sequence formed by the dimensions of each $H_j.$
\end{itemize}
The following corollary gives the controllability for the Cauchy problem associated to $E$-invariant operators.
\begin{corollary} Let $M$ be a compact manifold without boundary. Let $E$ be a positive classical pseudo-differential operator on $M$ and let us consider the operators $A,B:C^\infty(M)\rightarrow C^\infty(M)$ being $E$-invariant operators.
\begin{itemize}
        \item[(1)] If the Cauchy problem
    \begin{equation}\label{MainII:M} \begin{cases}\displaystyle\frac{d u}{dt}=Au+Bv,& 
    \\
\\u(0)=u_{0}\in C^{\infty}(M)
\text{ } \end{cases}
\end{equation}
    is controllable, then for any $\ell\in \mathbb{N}_0,$ the global symbols $\sigma_A(\ell)$ and $\sigma_B(\ell)$ of $A$ and $B$ associated to the spectral decomposition of $E,$ respectively, satisfy the Kalman condition:
    \begin{equation}\label{Rank:M}
     \textnormal{rank}\left[\sigma_B(\ell),\ \sigma_A(\ell)\sigma_B(\ell),\ \cdots,\ \sigma_A(\ell)^{{d_\ell}-1}\sigma_B(\ell)\right]=d_{\ell}.   
    \end{equation}Additionally, if $A$ generates a strongly continuous semigroup on $L^2(M),$ the image of the Cauchy problem \eqref{MainII:M} under the Fourier transform relative to the decomposition $(H_j=\textnormal{Ker}(E-\lambda_jI))_{j\in \mathbb{N}},$  has a {\it finite global controllability cost} at time $T>0,$ that is 
    $$  \mathscr{C}_{T}:=   \sup_{\ell\in \mathbb{N}_0 }\mathscr{C}_{\ell,T}<\infty.$$
    Moreover, 
    $$   \mathscr{C}_{T}\leq \Tilde{\mathscr{C}}_T, $$ where  $\tilde{\mathscr{C}_{T}}$ is  the controllability cost of \eqref{MainII:M}.
    \item[(2)] Conversely, assume that $A$ is the generator of a strongly continuous semigroup on $\mathcal{H},$ and that the Kalman condition \eqref{Rank:M} is satisfied for each $\ell\in\mathbb{N}_0$. Assume that the image of the Cauchy problem \eqref{MainII:M} under the Fourier transform relative to the spectral decomposition $(H_j=\textnormal{Ker}(E-\lambda_j I))_{j\in \mathbb{N}},$  has a {\it finite global controllability cost} at time $T>0,$ that is, $$  \mathscr{C}_{T}:=   \sup_{\ell\in \mathbb{N}_0 }\mathscr{C}_{\ell,T}<\infty. $$ Then, the Cauchy problem  \eqref{MainII:M} is controllable at time $T>0,$ and its controllability costs $\tilde{\mathscr{C}_{T}}$ satisfies the inequality
    \begin{equation}
      \mathscr{C}_{T}\geq   \tilde{\mathscr{C}_{T}}.
    \end{equation}
    \end{itemize}    
\end{corollary}
\begin{proof}
    The statement of this corollary follows from Theorem \ref{Main:Theorem:HS} applied to the $E$-invariant operators $A$ and $B$ which, equivalently, are Fourier multipliers relative to the spectral decomposition $H_j=\textnormal{Ker}(E-\lambda_jI).$ Note that in this case $\mathcal{H}^\infty=C^\infty(M)$ and $\mathcal{H}=L^2(M).$
\end{proof}

\subsection{Kalman criterion on compact Lie groups}\label{Compact:lie:group:sec}

In this subsection we prove our controllability criterion in the setting of compact Lie groups. First, we adapt  Definition \ref{Fourier:costs} in terms of the group Fourier transform of the group.
\begin{definition}\label{Cost:g} 
Let $A,B:C^\infty(G)\rightarrow C^\infty(G)$ be continuous and left-invariant  linear operators. Consider the Cauchy problem
\begin{equation}\label{Cauchy}
\textnormal{(CP)}:\begin{cases}\displaystyle\frac{d u}{dt}=Au+Bv,& 
    \\
\\u(0)=u_{0}\in C^{\infty}(G).
\text{ } \end{cases} 
\end{equation}We define the image of $\textnormal{(CP)}$ under the group Fourier transform to be the infinite family of finite-dimensional dynamical systems
\begin{equation}\label{Cauchy:xi}
(\textnormal{(CP)},{[\xi]}):\begin{cases}\displaystyle\frac{d\widehat{ u}(\xi)}{dt}=\sigma_A(\xi)\widehat{u}(\xi)+\sigma_B(\xi)\widehat{v}(\xi),& 
    \\
\\\widehat{u(0)}(\xi)=\widehat{u}_{0}(\xi)\in \mathbb{C}^{d_\xi\times d_\xi}.
\text{ } \end{cases},\,[\xi] \in \widehat{G}.
\end{equation}
\end{definition}
We present the following version of Lemma \ref{Lemma:ell} adapted to Definition \ref{Cost:g} for left-invariant operators. 
\begin{lemma}\label{Lemma:xi}The following statements are equivalent.
\begin{itemize}
\item[(1)] For all $\xi\in  \textnormal{Rep}(G),$ the Cauchy problem $(\textnormal{(CP)},{[\xi]})$ in  \eqref{Cauchy:xi} is a controllable dynamical system at a time $T>0$.
    \item[(2)] $\forall\xi\in\textnormal{Rep}(G)$, the Kalman condition 
    \begin{equation}
        \textnormal{rank}\left[\sigma_B(\xi),\ \sigma_A(\xi)\sigma_B(\xi),\ \cdots,\ \sigma_A(\xi)^{{d_\xi}-1}\sigma_B(\xi)\right]=d_{\xi}
    \end{equation}is satisfied.
    \item[(3)]  $\forall \xi\in\textnormal{Rep}(G),\ \exists c= c(\xi,T)>0$ such that \begin{equation}\label{Observability}
\displaystyle\int\limits_0^T||\sigma_B(\xi)^*\exp{(t\sigma_A(\xi)^*)}z||^2_{\textnormal{HS}}dt\geq c(\xi,T)^2||z||^2_{\textnormal{HS}},\ \forall z\in \mathbb{C}^{d_\xi\times d_\xi}.
    \end{equation}
\end{itemize}
\end{lemma}
\begin{proof} The proof follows  exactly the same steps as the one done above for Lemma  \ref{Lemma:ell}. Indeed,  the equivalence $\textnormal{(1)}\Longleftrightarrow \textnormal{(2)}$ follows from Kalman's criterion (see Theorem \ref{Kalmantheorem}). On the other hand, the equivalence  $\textnormal{(1)}\Longleftrightarrow \textnormal{(3)}$ is nothing else that the observability criterion in Theorem \ref{observabletheorem}. Is clear then the equivalence  $\textnormal{(2)}\Longleftrightarrow \textnormal{(3)}.$ The proof of Lemma \ref{Lemma:xi} is complete.    
\end{proof}
\begin{remark} Let $c_{\xi,T}$ be  the supremum of the constants $c= c(\xi,T)>0$ satisfying  the observability inequality \eqref{Observability}.
    According to the usual nomenclature of the control theory, the constant $$ \boxed{\mathscr{C}_{\xi,T}:={1}/{c_{\xi,T}}}$$   is called the controllability cost of the Cauchy problem  \eqref{Cauchy:xi}.
\end{remark}

\begin{definition}
    We will say that the image of the Cauchy problem \eqref{Cauchy} under the group Fourier transform has a {\it finite global controllability cost} if 
    \begin{equation}
     \mathscr{C}_{T}:=   \sup_{[\xi]\in\widehat{G} }\mathscr{C}_{\xi,T}<\infty.
    \end{equation}
\end{definition}
The following is the  main theorem of this subsection.
\begin{theorem}\label{theorem:compact:Lie:group}
    Let $G$ be a compact Lie group and let $A,B:C^\infty(G)\rightarrow C^\infty(G)$ be continuous left-invariant linear operators. 
    \begin{itemize}
        \item[(1)]  If the Cauchy problem
    \begin{equation}\label{Main} \begin{cases}\displaystyle\frac{d u}{dt}=Au+Bv,& 
    \\
\\u(0)=u_{0}\in C^{\infty}(G)
\text{ } \end{cases}
\end{equation}
    is controllable in time $T>0$. Then, for each representation space the global symbols $\sigma_A$ and $\sigma_B$ of $A$ and $B,$ respectively, satisfy the Kalman condition:
    \begin{equation}\label{Kalman:compact:lie:group}
      \textnormal{rank}\left[\sigma_B(\xi),\ \sigma_A(\xi)\sigma_B(\xi),\ \cdots,\ \sigma_A(\xi)^{{d_\xi}-1}\sigma_B(\xi)\right]=d_{\xi}. 
    \end{equation} Additionally, if   $A$ is the generator of a strongly continuous semigroup on $L^2(G),$ the image of the Cauchy problem \eqref{Cauchy} under the group Fourier transform has a {\it finite global controllability cost} at time $T>0,$ that is, $$  \mathscr{C}_{T}:=   \sup_{[\xi]\in\widehat{G} }\mathscr{C}_{\xi,T}<\infty. $$ Moreover,    if  $\tilde{\mathscr{C}_{T}}$ is the controllability costs of \eqref{Cauchy} then
    \begin{equation}
      \mathscr{C}_{T}\leq   \tilde{\mathscr{C}_{T}}.
    \end{equation}
    \item[(2)] Conversely, assume that $A$ is the generator of a strongly continuous semigroup on $L^2(G),$ and that the Kalman condition \eqref{Kalman:compact:lie:group} is satisfied for each $\xi\in\textnormal{Rep}(G)$. Assume that the image of the Cauchy problem \eqref{Cauchy} under the group Fourier transform has a {\it finite global controllability cost} at time $T>0,$ that is, $$  \mathscr{C}_{T}:=   \sup_{[\xi]\in\widehat{G} }\mathscr{C}_{\xi,T}<\infty. $$ Then, the Cauchy problem  \eqref{Cauchy} is controllable at time $T>0,$ and its controllability costs $\tilde{\mathscr{C}_{T}}$ satisfies the inequality
    \begin{equation}
      \mathscr{C}_{T}\geq   \tilde{\mathscr{C}_{T}}.
    \end{equation}
    \end{itemize}
\end{theorem}
\begin{proof} For the proof of $(1)$ note that a direct application of Theorem \ref{MainII} provides a Kalman condition of dimension $d_\xi^2$ in any representation space (see Remark \ref{main:remark:global}). To reduce the dimension we will use the Cayley-Hamiton theorem in each representation space.  We do it as follows.
\begin{itemize}
    \item  Proof of (1).   Assume that the Cauchy problem \eqref{Main} is controllable. By fixing $\xi\in\textnormal{Rep}(G)$ and taking the group Fourier transform in both sides, we get \begin{equation*}
    \displaystyle\frac{d\widehat{u}(\xi)}{dt}=\widehat{Au}(\xi)+\widehat{Bv}(\xi).\end{equation*}
   Since $A,B$  are left-invariant we have
    \begin{equation}\label{Maintransformed}
    \displaystyle\frac{d\widehat{u}(\xi)}{dt}=\sigma_A(\xi)\widehat{u}(\xi)+\sigma_B(\xi)\widehat{v}(\xi).
    \end{equation}
    This is a dynamical system in the set of square matrices of order $d_\xi$ which can be identified with $\mathbb{C}^{d_\xi^2}$ via 
    \begin{equation}\label{identification}
        \textbf{Z}=(z_{ij})_{d_\xi\times d_\xi}\longmapsto \left(\begin{array}{c}\textbf{Z}^{1}\\
    \vdots\\
    \textbf{Z}^{d_\xi}\end{array}\right),
    \end{equation}
    where $\textbf{Z}^{j}$ denotes de $j$-th column vector of the matrix $\textbf{Z}$; so that equation \eqref{Maintransformed} becomes
    \begin{equation}\label{Maintransformed2}
    \displaystyle\frac{d\widehat{U}(\xi)}{dt}=\Sigma_A(\xi)\widehat{U}(\xi)+\Sigma_B(\xi)\widehat{V}(\xi),
    \end{equation}
    with 
    \begin{align*}
    &\Sigma_C(\xi):=\left(\begin{array}{cccc}
    \sigma_C(\xi) & \textbf{0}_{d_\xi\times d_\xi} & \cdots & \textbf{0}_{d_\xi\times d_\xi}\\
    \textbf{0}_{d_\xi\times d_\xi}&\sigma_C(\xi)&\cdots& \textbf{0}_{d_\xi\times d_\xi}\\
    \vdots & \vdots & \ddots & \vdots\\
    \textbf{0}_{d_\xi\times d_\xi} & \textbf{0}_{d_\xi\times d_\xi} & \cdots & \sigma_C(\xi)\end{array}\right),\ C\in\{A,B\},\\
    \\
    &\widehat{U}(\xi):=\left(\begin{array}{c}\widehat{u}(\xi)^{1}\\
    \vdots\\
    \widehat{u}(\xi)^{d_\xi}\end{array}\right)\ \text{and}\ \widehat{V}(\xi):=\left(\begin{array}{c}\widehat{v}(\xi)^{1}\\
    \vdots\\
    \widehat{v}(\xi)^{d_\xi}\end{array}\right).
    \end{align*}
With the notation above, we have proved that the two systems 
 \begin{equation}\label{Maintransformed}
    \displaystyle\frac{d\widehat{u}(\xi)}{dt}=\sigma_A(\xi)\widehat{u}(\xi)+\sigma_B(\xi)\widehat{v}(\xi),\,\,\frac{d\widehat{U}(\xi)}{dt}=\Sigma_A(\xi)\widehat{U}(\xi)+\Sigma_B(\xi)\widehat{V}(\xi),
    \end{equation}are equivalent, in the sense that any solution of \eqref{Maintransformed} induces a solution of the system \eqref{Maintransformed2} and  vice versa.  Observe also that we can enumerate the unitary dual $\widehat{G}$ as $[\xi_{j}],$ for $j\in\mathbb{N}_0.$  In this way we fix the orthonormal basis
\begin{equation}
\{e_{jk}\}_{k=1}^{d_j}=\{ d_{\xi_{j}}^{\frac{1}{2}}(\xi_{j})_{il}  \}_{i,l=1}^{d_{\xi_j}}
\end{equation}
where $d_{j}=d_{\xi_j}^2.$ Then, we have the subspaces $$H_{j}=\textnormal{span}\{(\xi_j)_{i,l}:i,l=1,\cdots,d_{\xi_j}\}.$$
In view of Remark \ref{main:remark:global}, note that the symbols $\sigma_{A}(\ell)$ and $\sigma_{B}(\ell)$ of $A$ and $B$ relative to the decomposition  $H_{j}=\textnormal{span}\{(\xi_j)_{i,l}:i,l=1,\cdots,d_{\xi_j}\}$ are given by
\begin{equation}
    \sigma_A(\ell)\equiv \Sigma_A(\xi_\ell) \textnormal{  and  }\sigma_B(\ell)\equiv \Sigma_B(\xi_\ell),\,\,\ell\in \mathbb{N}_0,
\end{equation}respectively. Now, by applying  Theorem \ref{Main:Theorem:HS}, we deduce that our hypothesis (1)
    implies the Kalman condition
    $$\textnormal{rank}\left[\Sigma_B(\xi_\ell),\ \Sigma_A(\xi_\ell)\Sigma_B(\xi_\ell),\ \cdots,\ \Sigma_A(\xi_\ell)^{{d_{\xi_\ell}^2 } -1}\Sigma_B(\xi_\ell)\right]=d_{\xi_\ell}^2,$$
where we have used that for any $\ell,$ $d_\ell=d_{\xi_\ell}^2.$ Since the map $\ell\mapsto [\xi_\ell]$ is a bijection, we will omit the subscript $\ell$ in $\xi_\ell$ and we will return to the generic notation $\xi$ for a unitary representation $\xi\in \textnormal{Rep}(G).$
Now, let us make a refinement  of  the Kalman condition
$$\textnormal{rank}\left[\Sigma_B(\xi),\ \Sigma_A(\xi)\Sigma_B(\xi),\ \cdots,\ \Sigma_A(\xi)^{{d_\xi^2}-1}\Sigma_B(\xi)\right]=d_{\xi}^2.$$
For instance, note that 
    $$\Sigma_A(\xi)^j\Sigma_B(\xi)=\left(\begin{array}{cccc}
    \sigma_A(\xi)^j\sigma_B(\xi) & \textbf{0}_{d_\xi\times d_\xi} & \cdots & \textbf{0}_{d_\xi\times d_\xi}\\
    \textbf{0}_{d_\xi\times d_\xi}&\sigma_A(\xi)^j\sigma_B(\xi)&\cdots& \textbf{0}_{d_\xi\times d_\xi}\\
    \vdots & \vdots & \ddots & \vdots\\
    \textbf{0}_{d_\xi\times d_\xi} & \textbf{0}_{d_\xi\times d_\xi} & \cdots & \sigma_A(\xi)^j\sigma_B(\xi)\end{array}\right),\ j=0,...,d_\xi^2-1.$$
    Thus 
    \begin{align*}
    &\textnormal{rank}\left[\Sigma_B(\xi),\ \Sigma_A(\xi)\Sigma_B(\xi),\ \cdots,\ \Sigma_A(\xi)^{{d_\xi^2}-1}\Sigma_B(\xi)\right]=\\
    & d_{\xi}\cdot \textnormal{rank}\left[\sigma_B(\xi),\ \sigma_A(\xi)\sigma_B(\xi),\ \cdots,\ \sigma_A(\xi)^{{d_\xi}^2-1}\sigma_B(\xi)\right],
    \end{align*}
    and therefore
    $$\textnormal{rank}\left[\sigma_B(\xi),\ \sigma_A(\xi)\sigma_B(\xi),\ \cdots,\ \sigma_A(\xi)^{{d_\xi}^2-1}\sigma_B(\xi)\right]=d_\xi.$$
    Now, by the Caley-Hamilton theorem we know that if 
    $$  P_{\sigma_A(\xi)}(\lambda)=\det(\lambda I_{d_\xi\times d_\xi}-\sigma_A(\xi))=\lambda^{d_\xi}-\sum_{j=0}^{d_\xi-1}\alpha_{\xi,j}\lambda^j,$$
    is the characteristic polynomial of the symbol $\sigma_A(\xi),$ then $$  P_{\sigma_A(\xi)}(\sigma_A(\xi))=0,$$ which is equivalent to say that
    \begin{equation}
        \sigma_A(\xi)^{d_\xi}=\sum_{j=0}^{d_\xi-1}\alpha_{\xi,j}\sigma_A(\xi)^j.
    \end{equation}    
   In consequence,  for every $j\in\{d_\xi,...,d_\xi^2-1\},$ there exist $c_{j}^0,c_{j}^1,\cdots ,c_{j}^{j}\in\mathbb{C},$ where each $c_{k}^{j}=c_{k}^{j}(\xi)$ depends of $\xi,$ such that $$\sigma_A(\xi)^j=\sum\limits_{k=0}^{j}c_{j}^{k}\sigma_A(\xi)^k.$$  
   So in the matrix 
   $$  \left[\sigma_B(\xi),\ \sigma_A(\xi)\sigma_B(\xi),\ \cdots,\ \sigma_A(\xi)^{{d_\xi}-1}\sigma_B(\xi)\right],  $$
   the terms $\sigma_A(\xi)^{j}\sigma_B(\xi),$ $d_\xi\leq j\leq d_\xi^2-1,$ can be written as a linear combination of the first $d_\xi-1$ matrix-blocks $\sigma_A(\xi)^{j'}\sigma_B(\xi),$ $1\leq j'\leq d_\xi-1.$ After doing the Gaussian reduction we have that
    \begin{align*}
      &\textnormal{rank}\left[\sigma_B(\xi),\ \sigma_A(\xi)\sigma_B(\xi),\ \cdots,\ \sigma_A(\xi)^{{d_\xi}-1}\sigma_B(\xi)\right]=\\
      &\textnormal{rank}\left[\sigma_B(\xi),\ \sigma_A(\xi)\sigma_B(\xi),\ \cdots,\ \sigma_A(\xi)^{{d_\xi}^2-1}\sigma_B(\xi)\right]=d_\xi.
    \end{align*} The proof of (1) is complete. 
\item     
    Now, let us prove (2). So, 
    conversely, suppose that $$\forall\xi\in\textnormal{Rep}(G),\ \textnormal{rank}\left[\sigma_B(\xi),\ \sigma_A(\xi)\sigma_B(\xi),\ \cdots,\ \sigma_A(\xi)^{{d_\xi}-1}\sigma_B(\xi)\right]=d_\xi.$$
  We want to prove the controllability of the Cauchy problem \eqref{Main} at any time $T>0$. According to Theorem \ref{observabletheorem}, it is sufficient to show that there exists $c_T>0$ such that 
    \begin{equation}\label{observableinequality}
    \displaystyle\int\limits_0^T||B^*S(t)^*f||^2_{L^2(G)}dt\geq c_T^2||f||^2_{L^2(G)},\ \forall f\in L^2(G),
    \end{equation}
    where $S(t)$ is the $C_0$-semigroup generated by $A$. 
    By the Kalman criterion, we know that the system
    \begin{equation}
        \displaystyle\frac{d\gamma_{\xi}}{dt}+\sigma_A(\xi)\gamma_{\xi}=\sigma_B(\xi)v_{\xi}
    \end{equation}
     is controllable for every $\xi\in\textnormal{Rep}(G).$ In consequence the inequality 
    \begin{equation}\label{Observabilioty:dxi}
    \displaystyle\int\limits_0^T||\sigma_B(\xi)^*\exp{(t\sigma_A(\xi)^*)}z||^2_{\textnormal{HS}}dt\geq c_{\xi,T}^2||z||^2_{\textnormal{HS}},\ \forall z\in \mathbb{C}^{d_\xi\times d_\xi},
    \end{equation}holds and let us denote by $c_{\xi,T}>0$ the largest constant that satisfies this inequality. Now, by using the identification in \eqref{identification} we observe that the observability inequality in \eqref{Observabilioty:dxi} is equivalent to the following  observability inequality 
\begin{equation}\label{Observabilioty:ell}
    \displaystyle\int\limits_0^T||\Sigma_B(\xi)^*\exp{(t\Sigma_A(\xi)^*)}z||^2_{\mathbb{C}^{d_\xi^2}}dt\geq c_{\xi,T}^2||z||^2_{\mathbb{C}^{d_\xi^2}},\ \forall z\in \mathbb{C}^{d_\xi^2}.
    \end{equation} This analysis shows that the controllability costs of the systems
    \begin{equation}\label{Maintransformed}
    \displaystyle\frac{d\widehat{u}(\xi)}{dt}=\sigma_A(\xi)\widehat{u}(\xi)+\sigma_B(\xi)\widehat{v}(\xi),\,\,\frac{d\widehat{U}(\xi)}{dt}=\Sigma_A(\xi)\widehat{U}(\xi)+\Sigma_B(\xi)\widehat{V}(\xi),
    \end{equation}are the same. In other words, by following the lines in the proof of $\textnormal{(2)}$ in Theorem \ref{MainII},   we have that the following observability inequality \begin{equation*}\displaystyle\int\limits_0^T||B^*S(t)^*f||^2_{L^2(G)}dt\geq c_{T}^2
        ||f||^2_{L^2(G)},\,\,\forall f\in L^2(G),
    \end{equation*} holds with $$c_T:=\inf\limits_{[\xi]\in\widehat{G}}c_{\xi,T}.$$
    So the equality in \eqref{observableinequality} holds with $$c_T:=\inf\limits_{[\xi]\in\widehat{G}}c_{\xi,T}= \inf\limits_{[\xi]\in\widehat{G}}1/\mathscr{C}_{\xi,T}=  1/ \sup_{[\xi]\in\widehat{G} }\mathscr{C}_{\xi,T}<\infty.$$ The proof of (2) is complete. Indeed, note that the
    controllability cost $\tilde{\mathscr{C}_{T}}$ of \eqref{Main} is the infimum of the constants $c_T>0$ satisfying  \eqref{observableinequality}, from where we deduce that
    \begin{equation}
      \mathscr{C}_{T}\geq   \tilde{\mathscr{C}_{T}}.
    \end{equation} 
    \end{itemize}
    Having proved (1) and (2) the proof of Theorem \ref{theorem:compact:Lie:group} is complete.
    \end{proof}

\subsection{Controllability for  fractional subelliptic diffusion models}\label{section:applications:Hor} Let $G$ be a compact Lie group, $\mathbb{X}=\{X_1,...,X_k\}$ be an orthonormal set of real left-invariant vector fields satisfying the H\"ormander condition at step $r$, and let $s>0.$ The positive fractional sub-Laplacian on $G$ or order $s$ associated to $\mathbb{X}$ is the operator $$\mathcal{L}_s:=\left(-\sum\limits_{j=1}^kX_j^2\right)^{s/2}.$$
The symbol of $\mathcal{L}_s$ can be computed in terms of the symbol of the sub-Laplacian $\mathcal{L}=-\sum\limits_{j=1}^kX_j^2,$  
 $$\sigma_{ \mathcal{L}}(\xi)\equiv \textnormal{diag}(\lambda_{1,[\xi]},...,\lambda_{d_\xi,[\xi]}),\,\,\xi\in \textnormal{Rep}(G),$$ 
 as follows 
 \begin{equation}
     \sigma_{\mathcal{L}_s}(\xi)=\textnormal{diag}(\lambda_{1,[\xi]}^{s/2},...,\lambda_{d_\xi,[\xi]}^{s/2}),\,\, \xi\in \textnormal{Rep}(G).
 \end{equation} Recall that there exist constants $c,C>0$ such that the values $\lambda_{j,[\xi]}$ satisfy the inequality (see \cite{GarettoRuzhansky2015}) 
 \begin{equation}\label{eigenvalues:inequality}
 {{c\langle\xi\rangle^{1/r}}}\leq \lambda_{j,[\xi]}^{1/2}\leq C\langle\xi\rangle,\ \forall \xi\in\textnormal{Rep}(G),\ \forall j\in\{1,...,d_\xi\}.
 \end{equation}We have that $\sigma_{\mathcal{L}_s}(\xi)$ does not depend on  $x\in G,$ and that $\mathcal{L}_s$ is left-invariant. Let us consider the  {\it subelliptic diffusion model} 
 \begin{equation}\label{diffusionmodel}
 \frac{du}{dt}=-\mathcal{L}_su+Bv .  
\end{equation}
We shall prove the following:
\begin{itemize}
    \item[(1)] If \eqref{diffusionmodel} is controllable then $B:C^{\infty}(G)\rightarrow \textnormal{Im}(B)\subset  C^{\infty}(G)$ is invertible.
    \item[(2)] If $B:C^{\infty}(G)\rightarrow \textnormal{Im}(B)\subset  C^{\infty}(G)$ is invertible on its image subspace and its matrix-valued symbol satisfies the lower bound
    \begin{equation}\label{conditionforB}
      \forall z\in \mathbb{C}^{d_\xi\times d_\xi},\,  \Vert \sigma_{B}(\xi)^*z\Vert_{\textnormal{HS}}\geq C_B\langle \xi \rangle^\kappa \Vert z\Vert_{\textnormal{HS}}, 
    \end{equation}for some $\kappa\geq s/2,$ then  \eqref{diffusionmodel} is a controllable system.
\end{itemize}
For the proof of (1) let us proceed as follows. If \eqref{diffusionmodel} is controllable in any time $T>0$ then, by Theorem \ref{theorem:compact:Lie:group},
$$\textnormal{rank}[\sigma_B(\xi),\sigma_{-\mathcal{L}_s}(\xi)\sigma_B(\xi),\cdots,\sigma_{-\mathcal{L}_s}(\xi)^{d_\xi-1}\sigma_B(\xi)]=d_\xi,\ \forall\xi\in\textnormal{Rep}(G).$$ But $$\sigma_{-\mathcal{L}_s}(\xi)=\sigma_{-\mathcal{L}_s}(x,\xi)=\xi(x)^*(-\mathcal{L}_s\xi(x))=-(\xi(x)^*\mathcal{L}_s\xi(x))=-\sigma_{\mathcal{L}_s}(x,\xi)=-\sigma_{\mathcal{L}_s}(\xi)$$
(see Theorem \ref{The:quantisation:thm}), hence
$$\textnormal{rank}[\sigma_B(\xi),-\sigma_{\mathcal{L}_s}(\xi)\sigma_B(\xi),\cdots,[-\sigma_{\mathcal{L}_s}(\xi)]^{d_\xi-1}\sigma_B(\xi)]=d_\xi,\ \forall\xi\in\textnormal{Rep}(G).$$ Since $\sigma_{\mathcal{L}_s}(\xi)$ is a diagonal matrix we can show, by doing Gaussian reduction, that $$\textnormal{rank}[\sigma_B(\xi),-\sigma_{\mathcal{L}_s}(\xi)\sigma_B(\xi),\cdots,[-\sigma_{\mathcal{L}_s}(\xi)]^{d_\xi-1}\sigma_B(\xi)]=\textnormal{rank}[\sigma_B(\xi)],$$
so $\textnormal{rank}[\sigma_B(\xi)]=d_\xi,\ \forall\xi\in\textnormal{Rep}(G),$ i.e., $\sigma_B(\xi)$ is invertible for all $\xi\in\textnormal{Rep}(G).$ The formula $$B^{-1}f(x)=\sum_{[\xi]\in \widehat{G}}d_\xi\textnormal{Tr}[\xi(x)\sigma_{B}(\xi)^{-1}\widehat{f}(\xi)]$$ defines the inverse $B^{-1}:\textnormal{Im}(B)\rightarrow C^{\infty}(G)$ of $B:C^{\infty}(G)\rightarrow \textnormal{Im}(B)\subset  C^{\infty}(G).$
For the proof of $\textnormal{(2)}$ let us proceed as follows.
Note that the constant
\begin{equation}
    \mathscr{X}_{\xi,T}^2=\inf_{z\neq 0}\frac{ \displaystyle\int_0^T||\sigma_B(\xi)^*\exp{(t\sigma_{-\mathcal{L}_s}(\xi)^*)}z||^2_{\textnormal{HS}}dt }{||z||^2_{\textnormal{HS}}}, 
\end{equation}satisfies the inequality
\begin{equation}\label{Auxiliar:ineq}
    \int\limits_0^T||\sigma_B(\xi)^*\exp{(t\sigma_{-\mathcal{L}_s}(\xi)^*)}z||^2_{\textnormal{HS}}dt \geq \mathscr{X}_{\xi,T}^2||z||^2_{\textnormal{HS}},\,\,\forall z\neq 0.
\end{equation}Let $c_{\xi,T}^2\geq \mathscr{X}_{\xi,T}^2$ be the largest constant satisfying the inequality in \eqref{Auxiliar:ineq}. Note that
\begin{align*}
 c_{\xi,T}^2\geq \mathscr{X}_{\xi,T}^2 &= \inf_{z\neq 0}\frac{ \displaystyle\int_0^T||\sigma_B(\xi)^*\exp{(-t\sigma_{\mathcal{L}_s}(\xi)^*)}z||^2_{\textnormal{HS}}dt }{||z||^2_{\textnormal{HS}}}\\
 &\geq\inf_{z\neq 0}\frac{ \displaystyle\int_0^T C_B^2\langle\xi\rangle^{2\kappa} \left|\left|\exp{\left[-t\cdot\text{diag}\left(\lambda_{1,[\xi]}^{s/2},...,\lambda_{d_\xi,[\xi]}^{s/2}\right)\right]}z\right|\right|_{\textnormal{HS}}^2dt}{||z||^2_{\textnormal{HS}}}\ \textnormal{(by \eqref{conditionforB})},\\
  &=\inf_{z\neq 0}\frac{ \displaystyle\int_0^T C_B^2\langle\xi\rangle^{2\kappa} \left|\left|\text{diag}\left(\exp{(-t\lambda_{1,[\xi]}^{s/2})},...,\exp{(-t\lambda_{d_\xi,[\xi]}^{s/2})}\right)z\right|\right|_{\textnormal{HS}}^2dt}{||z||^2_{\textnormal{HS}}}\\
  &\geq\inf_{z\neq 0}\frac{ \displaystyle\int_0^T C_B^2\langle\xi\rangle^{2\kappa} \exp{(-2t\gamma_{[\xi]}^{s/2})}\left|\left|z\right|\right|_{\textnormal{HS}}^2dt}{||z||^2_{\textnormal{HS}}}\ \textnormal{(where } \gamma_{[\xi]}:=\max\limits_{1\leq j\leq d_\xi}\lambda_{j,[\xi]}\textnormal{)},\\
  &\geq\inf_{z\neq 0}\frac{ \displaystyle\int_0^T C_B^2\langle\xi\rangle^{2\kappa} \exp{(-2tC\langle\xi\rangle^{s})}\left|\left|z\right|\right|_{\textnormal{HS}}^2dt}{||z||^2_{\textnormal{HS}}}\ \textnormal{(by \eqref{eigenvalues:inequality})},\\
  &\displaystyle=C_B^2\langle\xi\rangle^{2\kappa}\times\frac{\left(1-e^{-2TC\langle\xi\rangle^{s}}\right)}{2C\langle\xi\rangle^{s}}\\
  &\displaystyle=C_B^2\langle\xi\rangle^{2\kappa-s}\left(1-e^{-2TC\langle\xi\rangle^{s}}\right)
\end{align*}
 In consequence, since $\langle\xi\rangle\geq 1$ and $\kappa\geq s/2,$ we have that 
\begin{equation}
   c_{\xi,T}^2\geq  C_B^2\langle\xi\rangle^{2\kappa-s}\left(1-e^{-2TC\langle\xi\rangle^{s}}\right)\geq C_B^2\left(1-e^{-2TC}\right)  \neq 0.
\end{equation}All the previous analysis shows that
\begin{equation}
    \inf_{[\xi]\in \widehat{G}}c_{\xi,T}^2\geq C_B^2\left(1-e^{-2TC}\right)\neq 0.
\end{equation}
With the notation of the proof of Theorem \ref{theorem:compact:Lie:group} we have that
$$c_T:=\inf\limits_{[\xi]\in\widehat{G}}c_{\xi,T}= \inf\limits_{[\xi]\in\widehat{G}}1/\mathscr{C}_{\xi,T}=  1/ \sup_{[\xi]\in\widehat{G} }\mathscr{C}_{\xi,T}<\infty.$$ 
Then, we have proved that   \eqref{diffusionmodel} is a controllable system. Note that in this case the controllability cost $\Tilde{\mathscr{C}}_T$ of \eqref{diffusionmodel} can be estimated as
\begin{align*}
  \Tilde{\mathscr{C}}_T^2\leq   \mathscr{C}_T^2=\sup_{[\xi]\in\widehat{G} }\mathscr{C}_{\xi,T}^2=\inf_{[\xi]\in \widehat{G}}1/c_{\xi,T}^2\leq \frac{1}{C_B^2\left(1-e^{-2TC}\right)}.
\end{align*}In consequence,
\begin{align*}
    \Tilde{\mathscr{C}}_T\leq \frac{1}{C_B\sqrt{\left(1-e^{-2TC}\right)}}.
\end{align*}

Summarising all the discussion above we have proved the following controllability criterion for subelliptic diffusions models on $G.$
\begin{theorem}
    Let $G$ be a compact Lie group,  and let $B:C^\infty(G)\rightarrow C^\infty(G)$ be a continuous left-invariant linear operator. 
    \begin{itemize}
        \item[(1)]  If the Cauchy problem
    \begin{equation}\label{Main:example} \begin{cases}\displaystyle \frac{du}{dt}=-\mathcal{L}_su+Bv  ,& 
    \\
\\u(0)=u_{0}\in C^{\infty}(G)
\text{ } \end{cases}
\end{equation}
    is controllable in time $T>0$, then $B:C^{\infty}(G)\rightarrow \textnormal{Im}(B)\subset  C^{\infty}(G)$ is an invertible continuous linear operator on its image.
    \item[(2)] Conversely, assume  $B:C^{\infty}(G)\rightarrow \textnormal{Im}(B)\subset  C^{\infty}(G)$ is invertible on its image subspace and that its matrix-valued symbol satisfies the lower bound
    \begin{equation}
      \forall z\in \mathbb{C}^{d_\xi\times d_\xi},\,  \Vert \sigma_{B}(\xi)^*z\Vert_{\textnormal{HS}}\geq C_B\langle \xi \rangle^\kappa \Vert z\Vert_{\textnormal{HS}}, 
    \end{equation}for some $\kappa\geq s/2.$ Then  \eqref{Main:example} is a controllable system and its controllability cost $\Tilde{\mathscr{C}}_T$ satisfies the estimate
    \begin{equation}
         \Tilde{\mathscr{C}}_T\leq \frac{1}{C_B\sqrt{\left(1-e^{-2TC}\right)}},
    \end{equation} for some $C>0$.
    \end{itemize}
\end{theorem}
\subsection{Wave equation vs. heat equation on Hilbert spaces}\label{Wave:vs:heat} Let $\mathcal{H}$ be a complex Hilbert space and let $\mathcal{H}^{\infty}\subset \mathcal{H}$ be a dense
linear subspace of $\mathcal{H}$. Let $\mathcal{H}=\bigoplus_{j}H_j$ be a decompositon of $\mathcal{H}$ in orthogonal subspaces $H_j$ of dimension $d_{j}\in\mathbb{N}.$ Let $A,B:\mathcal{H}^\infty\rightarrow \mathcal{H}$ be Fourier multipliers relative to the decomposition $\{H_j\}_{j\in \mathbb{N}_0}.$     Let us consider the second order Cauchy problem
\begin{equation}\label{wave:cauchy:problem}
    \left\{\begin{array}{l}
        \displaystyle\frac{d^2u}{dt^2}=Au+Bv;\\
        \\
        u(0)=u_0,\ u_t(0)=\tilde{u}_0,
        \end{array}\right.
\end{equation}
where $u:[0,T]\rightarrow \mathcal{H}^\infty$ is of $C^2$-class in time. The differential equation in \eqref{wave:cauchy:problem} is a {\it wave equation} and, analogously to the case of a first-order control system, we say that it is controllable in time $T>0$ if for every $u_T,\tilde{u}_T\in  \mathcal{H}^\infty$, there exists a control $v:[0,T]\rightarrow  \mathcal{H}^\infty$ such that the solution of \eqref{wave:cauchy:problem} satisfies $u(T)=u_T$ and $u_t(T)=\tilde{u}_T.$ 

On the other hand, the first order differential equation in \eqref{MainII} is a {\it heat equation}. It is well known that in the case of internal control (i.e. when $B$ is given by the multiplication operator $Bv=1_{\omega}v$ where $\omega$ is an open subset of $G$) the controllability of \eqref{wave:cauchy:problem} implies the controllability of \eqref{Main}. We refer to Kannai \cite{Kannai}, Russell \cite{Russell73}, and Miller \cite{Miller2004} for details. We shall prove the same result for any left-invariant operator $B$ satisfying some conditions. More precisely, we have the following theorem.

\begin{theorem}\label{thorem:wave:heat} Let $A,B:\mathcal{H}^\infty\rightarrow \mathcal{H}$ be Fourier multipliers relative to the decomposition $\{H_j\}_{j\in \mathbb{N}_0}$ of a Hilbert space $\mathcal{H}.$ Assume that $A$ is the generator of a strongly continuous semigroup. If the second order Cauchy problem 
\begin{equation}\label{wave:cauchy:problem:Th}
    \left\{\begin{array}{l}
        \displaystyle\frac{d^2u}{dt^2}=Au+Bv;\\
        \\
        u(0)=u_0,\ u_t(0)=\tilde{u}_0;\,u_0,\Tilde{u}_0\in \mathcal{H}^\infty,
        \end{array}\right.
\end{equation} is controllable in time $T>0$ then first order Cauchy problem 
 \begin{equation}\label{Main:tw}\begin{cases}\displaystyle\frac{d u}{dt}=Au+Bv,& 
    \\
\\u(0)=u_{0}\in \mathcal{H}^\infty
\text{ } \end{cases}
\end{equation}
is controllable in time $T>0$ provided that the image of the Cauchy problem \eqref{Main:tw} under the Fourier transform relative to the decomposition $(H_j)_{j\in \mathbb{N}},$  has a {\it finite global controllability cost} at time $T>0$.  
\end{theorem}

\begin{proof} First, we will make a reduction of order by setting $u_1:=u$ and $u_2:=u_t$ so that the Cauchy problem \eqref{wave:cauchy:problem:Th} becomes 
\begin{equation}\label{wave:reduced}
    \left\{\begin{array}{l}
    \displaystyle\frac{d}{dt}\left(\begin{array}{c}
    u_1\\
    u_2
    \end{array}\right)=\left(\begin{array}{cc}
    \textbf{0} & \text{Id}\\
    A & \textbf{0}\end{array}\right)\left(\begin{array}{c}
    u_1\\
    u_2
    \end{array}\right)+\left(\begin{array}{cc}
    \textbf{0} & \textbf{0}\\
    \textbf{0} & B\end{array}\right)\left(\begin{array}{c}
    w\\
    v
    \end{array}\right),\\
    \\
    \left(\begin{array}{c}
    u_1(0)\\
    u_2(0)
    \end{array}\right)=\left(\begin{array}{c}
    u_0\\
    \tilde{u}_0
    \end{array}\right).
    \end{array}\right.
\end{equation}
It is clear that the notion of controllability that we defined for \eqref{wave:cauchy:problem:Th} is equivalent to the definition of controllability of the first order system \eqref{wave:reduced}. Let us  suppose that \eqref{wave:reduced} is controllable in time $T>0.$  We shall show that the finite-dimensional control system \begin{equation}\label{wave:transformed}
    \left\{\begin{array}{l}
    \displaystyle\frac{d}{dt}\left(\begin{array}{c}
    \widehat{u}_1(\ell)\\
    \widehat{u}_2(\ell)
    \end{array}\right)=\left(\begin{array}{cc}
    \textbf{0} & \text{I}\\
    \sigma_A(\ell) & \textbf{0}\end{array}\right)\left(\begin{array}{c}
    \widehat{u}_1(\ell)\\
    \widehat{u}_2(\ell)
    \end{array}\right)+\left(\begin{array}{cc}
    \textbf{0} & \textbf{0}\\
    \textbf{0} & \sigma_B(\ell)\end{array}\right)\left(\begin{array}{c}
    \widehat{w}(\ell)\\
    \widehat{v}(\ell)
    \end{array}\right),\\
    \\
    \left(\begin{array}{c}
    \widehat{u}_1(\ell)(0)\\
    \widehat{u}_2(\ell)(0)
    \end{array}\right)=\left(\begin{array}{c}
    \widehat{u_0}(\ell)\\
    \widehat{\tilde{u}}_0(\ell)
    \end{array}\right),
    \end{array}\right.
\end{equation}
where $ \widehat{u}_j(\ell),\sigma_A(\ell),I,\sigma_B(\ell),\widehat{w}(\ell),\widehat{v}(\ell)\in\mathbb{C}^{d_\ell\times d_\ell}$, is also controllable in time $T.$ In fact, let $\left(\begin{array}{c}
    \zeta_{1,T}\\
    \zeta_{2,T}
    \end{array}\right)\in\mathbb{C}^{2d_\ell\times d_\ell}$, then the Fourier inversion formula implies that the functions $$u_T:= (e_\ell,\zeta_{1,T})_{\mathbb{C}^{d_\ell}}\ \textnormal{and}\ \tilde{u}_T:=(e_\ell,\zeta_{2,T})_{\mathbb{C}^{d_\ell}}$$ belong to $\mathcal{H}^\infty$ and $\widehat{u_T}(\ell)=\zeta_{1,T}$, $\widehat{\tilde{u}}_T(\ell)=\zeta_{2,T}.$ Since \eqref{wave:reduced} is controllable, there exist $w,v:[0,T]\rightarrow \mathcal{H}^\infty$ such that the solution $\left(\begin{array}{c}
    u_1\\
    u_2
    \end{array}\right)$ of \eqref{wave:reduced} is such that $u_1(T)=u_T$ and $u_2(T)=\tilde{u}_T.$
By taking the Fourier transform in \eqref{wave:reduced} at $\ell\in\mathbb{N}$ we obtain that $\left(\begin{array}{c}
    \widehat{u}_1(\ell)\\
    \widehat{u}_2(\ell)
    \end{array}\right)$ is the solution of \eqref{wave:transformed} and $\left(\begin{array}{c}
    \widehat{u}_1(\ell)(T)\\
    \widehat{u}_2(\ell)(T)
    \end{array}\right)=\left(\begin{array}{c}
    \zeta_{1,T}\\
    \zeta_{2,T}
    \end{array}\right)$. This argument holds for any $\left(\begin{array}{c}
    \zeta_{1,T}\\
    \zeta_{2,T}
    \end{array}\right)$, thus \eqref{wave:transformed} is controllable in time $T.$ Now, we can apply the rank Kalman condition to conclude that
    \begin{equation*}
    \textnormal{rank}\left[\left(\begin{array}{cc}
    \textbf{0} & \text{I}\\
    \sigma_A(\ell) & \textbf{0}\end{array}\right)^j\left(\begin{array}{cc}
    \textbf{0} & \textbf{0}\\
    \textbf{0} & \sigma_B(\ell) \end{array}\right)\right]_{0\leq j\leq 2d_\ell-1}=2d_\ell,
    \end{equation*}
    but 
    \begin{equation*}
        \left(\begin{array}{cc}
    \textbf{0} & \text{I}\\
    \sigma_A(\ell) & \textbf{0}\end{array}\right)^j\left(\begin{array}{cc}
    \textbf{0} & \textbf{0}\\
    \textbf{0} & \sigma_B(\ell) \end{array}\right)=\left\{\begin{array}{lll}
    \left(\begin{array}{cc}
    \textbf{0} & \textbf{0}\\
    \textbf{0} & \sigma_A(\ell)^{\frac{j}{2}}\sigma_B(\ell) \end{array}\right), &&\text{if}\ j\ \text{is even,}\\
    \\
    \left(\begin{array}{cc}
    \textbf{0} & \sigma_A(\ell)^{\frac{j-1}{2}}\sigma_B(\ell) \\
    \textbf{0} & \textbf{0}\end{array}\right), &&\text{if}\ j\ \text{is odd},\end{array}\right.
    \end{equation*}
    thus
    \begin{align*}
    2d_\ell=&\ \textnormal{rank}\left[\left(\begin{array}{cc}
    \textbf{0} & \text{I}\\
    \sigma_A(\ell) & \textbf{0}\end{array}\right)^j\left(\begin{array}{cc}
    \textbf{0} & \textbf{0}\\
    \textbf{0} & \sigma_B(\ell) \end{array}\right)\right]_{0\leq j\leq 2d_\ell-1}\\
    \\
    =&\ \textnormal{rank}\left[\left(\begin{array}{cc}
    \textbf{0} & \textbf{0}\\
    \textbf{0} & \sigma_A(\ell)^{i}\sigma_B(\ell) \end{array}\right),\left(\begin{array}{cc}
    \textbf{0} & \sigma_A(\ell)^{i}\sigma_B(\ell)\\
    \textbf{0} & \textbf{0}
    \end{array}\right)\right]_{0\leq i\leq d_\ell-1}\\
    \\
    =&\ 2\cdot\text{rank}[\sigma_B(\ell),\sigma_A(\ell)\sigma_B(\ell),\cdots,\sigma_A(\ell)^{d_\ell-1}\sigma_B(\ell)]\\
    \\
    \Longrightarrow&\ \text{rank}[\sigma_B(\ell),\sigma_A(\ell)\sigma_B(\ell),\cdots,\sigma_A(\ell)^{d_\ell-1}\sigma_B(\ell)]=d_\ell.
    \end{align*}
    This means that the rank Kalman condition \eqref{Kalman:compact:lie:group} is satisfied for each $\ell\in\mathbb{N}_0$. Additionally, if $A$ is the generator of a strongly continuous semigroup and the image of the Cauchy problem \eqref{Main:tw} under the group Fourier transform has finite global controllability cost, then by Theorem \ref{theorem:compact:Lie:group} (2), the system \eqref{Main:tw} is controllable. This completes the proof.
\end{proof}

As a consequence of Theorem \ref{thorem:wave:heat}, we consider the following application to the case of compact Lie groups.

\begin{corollary}\label{thorem:wave:heat:compact} Let $G$ be a compact Lie group, and $A,B:C^{\infty}(G)\longrightarrow C^{\infty}(G)$ be continuous left-invariant linear operators such that $A$ is the generator of a strongly continuous semigroup. If the second order Cauchy problem 
\begin{equation}
    \left\{\begin{array}{l}
        \displaystyle\frac{d^2u}{dt^2}=Au+Bv;\\
        \\
        u(0)=u_0,\ u_t(0)=\tilde{u}_0;\,u_0,\Tilde{u}_0\in C^\infty(G),
        \end{array}\right.
\end{equation} is controllable in time $T>0$. Then, the first order Cauchy problem 
 \begin{equation}
 \begin{cases}\displaystyle\frac{d u}{dt}=Au+Bv,& 
    \\
\\u(0)=u_{0}\in C^{\infty}(G)
\text{ } \end{cases}
\end{equation}
is controllable in time $T>0$, provided that its image under the group Fourier transform has finite global controllability cost.  
\end{corollary}
\begin{proof} For the proof let us use the notation in Remark  \ref{main:remark:global}. We can enumerate the unitary dual
    $\widehat{G}$ as $[\xi_{j}],$ for $j\in\mathbb{N}_0.$  In this way we fix the orthonormal basis
\begin{equation}
\{e_{jk}\}_{k=1}^{d_j}=\{ d_{\xi_{j}}^{\frac{1}{2}}(\xi_{j})_{il}  \}_{i,l=1}^{d_{\xi_j}},
\end{equation}
where $d_{j}=d_{\xi_j}^2.$ Then, we have the subspaces $$H_{j}=\textnormal{span}\{(\xi_j)_{i,l}:i,l=1,\cdots,d_{\xi_j}\}.$$
In view of Remark \ref{main:remark:global}, note that the symbols $\sigma_{A}(\ell)$ and $\sigma_{B}(\ell)$ of $A$ and $B$ relative to the decomposition  $H_{j}=\textnormal{span}\{(\xi_j)_{i,l}:i,l=1,\cdots,d_{\xi_j}\}$ are given by
\begin{equation}
    \sigma_A(\ell)\equiv \Sigma_A(\xi_\ell) \textnormal{  and  }\sigma_B(\ell)\equiv \Sigma_B(\xi_\ell),\,\,\ell\in \mathbb{N}_0,
\end{equation}respectively. The statement of Corollary \ref{thorem:wave:heat:compact} follows from Theorem \ref{thorem:wave:heat}. 
\end{proof}

\begin{corollary} In the context of Theorem \ref{thorem:wave:heat}, assume that the operator $$\Tilde{A}=\left(\begin{array}{cc}
    \textnormal{\textbf{0}} & \textnormal{Id}\\
    A & \textnormal{\textbf{0}}\end{array}\right)$$ is the infinitesimal generator of a $C_0$-semigroup and that  the rank Kalman condition 
    \begin{equation}\label{Kalman000}
        \text{rank}[\sigma_B(\ell),\sigma_A(\ell)\sigma_B(\ell),\cdots,\sigma_A(\ell)^{d_\ell-1}\sigma_B(\ell)]=d_\ell,
        \end{equation}
        holds for every $\ell\in\mathbb{N}_0$. Then the wave equation $d^2u/dt^2=Au+Bv$ is controllable in any time $T>0$ provided that
    \begin{equation}\label{wave:cost:condition}
    \displaystyle\inf\limits_{\ell\in\mathbb{N}_0}\sup\limits_{(z_1,z_2)\neq (0,0)}\frac{\displaystyle\int_0^T\left|\left|\sigma_B(\ell)^*S_1(t)\sigma_A(\ell)^*z_1+\sigma_B(\ell)^*S_2(t)z_2\right|\right|^2_{\textnormal{HS}}dt}{||z_1||_{\textnormal{HS}}^2+||z_2||_{\textnormal{HS}}^2}>0,
    \end{equation}
    where $z_1,z_2\in\mathbb{C}^{d_\ell}$, and  $$S_1(t):=\sum\limits_{n=0}^\infty \frac{t^{2n+1}}{(2n+1)!}(\sigma_A(\ell)^*)^{n},\textnormal{  and  }S_2(t):=\sum\limits_{n=0}^\infty \frac{t^{2n}}{(2n)!}(\sigma_A(\ell)^*)^{n},\,\,\,0\leq t\leq T.$$
    \end{corollary}
    \begin{proof} Note that the rank Kalman condition \eqref{Kalman000} is equivalent to the following Kalman conditon (as we have established in the proof of Theorem  \ref{thorem:wave:heat})
     \begin{align*}
    &\textnormal{rank}\left[\left(\begin{array}{cc}
    \textbf{0} & \text{I}\\
    \sigma_A(\ell) & \textbf{0}\end{array}\right)^j\left(\begin{array}{cc}
    \textbf{0} & \textbf{0}\\
    \textbf{0} & \sigma_B(\ell) \end{array}\right)\right]_{0\leq j\leq 2d_\ell-1}\\
    &=2d_\ell\\
    &=2\text{rank}[\sigma_B(\ell),\sigma_A(\ell)\sigma_B(\ell),\cdots,\sigma_A(\ell)^{d_\ell-1}\sigma_B(\ell)]
    \end{align*} and note that the observability inequality for the system \eqref{wave:transformed} reduces to 
    \begin{equation*}
        \int_0^T\left|\left|\left(\begin{array}{cc}
    \textbf{0} & \textbf{0}\\
    \textbf{0} & \sigma_B(\ell)^*\end{array}\right)\exp{\left[t\left(\begin{array}{cc}
    \textbf{0} & \text{I}\\
    \sigma_A(\ell)^* & \textbf{0} \end{array}\right)\right]}\left(\begin{array}{c}z_1\\z_2\end{array}\right)\right|\right|^2_{\textnormal{HS}}dt\geq c_{\ell,T}^2\left|\left|\left(\begin{array}{c}z_1\\z_2\end{array}\right)\right|\right|^2_{\textnormal{HS}}.
    \end{equation*}
    Since $$\left(\begin{array}{cc}
    \textbf{0} & \text{I}\\
    \sigma_A(\ell)^* & \textbf{0} \end{array}\right)^j=\left\{\begin{array}{lll}\left(\begin{array}{cc}
    (\sigma_A(\ell)^*)^{\frac{j}{2}} & \textbf{0}\\
    \textbf{0} & (\sigma_A(\ell)^*)^{\frac{j}{2}} \end{array}\right), &&\textnormal{if}\ j\ \textnormal{is even},\\
    \\
    \left(\begin{array}{cc}
    \textbf{0} & (\sigma_A(\ell)^*)^{\frac{j-1}{2}}\\
    (\sigma_A(\ell)^*)^{\frac{j+1}{2}} & \textbf{0} \end{array}\right), &&\textnormal{if}\ j\ \textnormal{is odd},\end{array}\right.$$
    then we have that $$\exp{\left[t\left(\begin{array}{cc}
    \textbf{0} & \text{I}\\
    \sigma_A(\ell)^* & \textbf{0} \end{array}\right)\right]}=\left(\begin{array}{cc}
    S_2(t) & S_1(t)\\
    S_1(t)\sigma_A(\ell)^* & S_2(t) \end{array}\right)$$
    so, the observability inequality above becomes
    \begin{equation*}
        \int_0^T\left|\left|\left(\begin{array}{cc}
    \textbf{0} & \textbf{0}\\
    \textbf{0} & \sigma_B(\ell)^*\end{array}\right)\left(\begin{array}{cc}
    S_2(t) & S_1(t)\\
    S_1(t)\sigma_A(\ell)^* & S_2(t) \end{array}\right)\left(\begin{array}{c}z_1\\z_2\end{array}\right)\right|\right|^2_{\textnormal{HS}}dt\geq c_{\ell,T}^2\left(||z_1||_{\textnormal{HS}}^2+||z_2||_{\textnormal{HS}}^2\right),
    \end{equation*}
    or equivalently,
    \begin{equation*}
    \displaystyle\int_0^T\left|\left|\sigma_B(\ell)^*S_1(t)\sigma_A(\ell)^*z_1+\sigma_B(\ell)^*S_2(t)z_2\right|\right|^2_{\textnormal{HS}}dt\geq c_{\ell,T}^2\left(||z_1||_{\textnormal{HS}}^2+||z_2||_{\textnormal{HS}}^2\right).
    \end{equation*}
    From here we can see that the condition \eqref{wave:cost:condition} is nothing but the property that  the Fourier transform of the system \eqref{wave:reduced} relative to the decomposition $(H_j)_{j\in\mathbb{N}_0}$  has finite global controllability cost. Therefore, by Theorem \ref{Main:Theorem:HS}, the system \eqref{wave:reduced} is controllable, so is \eqref{wave:cauchy:problem}. 
    \end{proof}
    \begin{remark} It is known that in the case of a compact Riemannian manifold $(M,g)$, if $A=\Delta$, where $\Delta$ is the negative Laplacian to the metric $g$,  the operator $$\Tilde{\Delta}=\left(\begin{array}{cc}
    \textbf{0} & \text{Id}\\
    \Delta & \textbf{0}\end{array}\right)$$ is the infinitesimal generator of a strongly continuous semigroup. This is due to the fact that it is  dissipative with respect to  a specific inner product defined in terms of the metric and the gradient of the manifold (see the classical work of Chen and Millman \cite{ChenMillman1980} for details). 
\end{remark}

\subsection{Control of the Schr\"odinger equation on Hilbert spaces}\label{Sch:section}
Let us consider the  Schr\"odinger equation
\begin{equation}\label{Schrodinger:problem}
    \left\{\begin{array}{l}
        \displaystyle i\frac{du}{dt}=Au+Bv;\\
        \\
        u(0)=u_0,
        \end{array}\right.
\end{equation}
where $u:[0,T]\rightarrow \mathcal{H}^\infty$ is of $C^2$-class in time, and  $A,B:\mathcal{H}^\infty\rightarrow \mathcal{H}$ are Fourier multipliers relative to the decomposition $\{H_j\}_{j\in \mathbb{N}_0}$ of a Hilbert space $\mathcal{H}.$ It is clear that \eqref{Schrodinger:problem} is equivalent to the following Cauchy problem
\begin{equation}\label{Schrodinger:problem:2}
    \left\{\begin{array}{l}
        \displaystyle \frac{du}{dt}=-iAu-iBv;\\
        \\
        u(0)=u_0.
        \end{array}\right.
\end{equation}Moreover, since $\sigma_{-iA}(\ell)=-i\sigma_A(\ell)$ and $\sigma_{-iB}(\ell)=-i\sigma_B(\ell),$ for all $\ell \in \mathbb{N},$
the Kalman condition
    \begin{equation}\label{Kalman:1}
       \forall \ell\in \mathbb{N},\,
       \textnormal{rank}\left[\sigma_{-iB}(\ell),\ \sigma_{-iA}(\ell)(\ell)\sigma_{-iB}(\ell),\ \cdots,\ \sigma_{-iA}(\ell)(\ell)^{{d_\ell}-1}\sigma_{-iB}(\ell)\right]=d_{\ell},
    \end{equation} 
holds if and only if the Kalman condition
  \begin{equation}\label{Kalman:2}
       \forall \ell\in \mathbb{N},\, \textnormal{rank}\left[\sigma_B(\ell),\ \sigma_A(\ell)\sigma_B(\ell),\ \cdots,\ \sigma_A(\ell)^{{d_\ell}-1}\sigma_B(\ell)\right]=d_{\ell},
    \end{equation} is satisfied. In view of the discussion above we have the following consequence of Theorem \ref{Main:Theorem:HS} for Schr\"odinger type models. Note that the natural assumption is that the operator $-iA$ generates a $C_0$-semigroup, which happens if for example the operator $A:\mathcal{H}^\infty\rightarrow \mathcal{H} $ is an unbounded  self-adjoint operator, see e.g. \cite{Weid}.
\begin{corollary}\label{Main:Theorem:HS:Sch}
Let $\mathcal{H}$ be a complex Hilbert space and let $\mathcal{H}^{\infty}\subset \mathcal{H}$ be a dense
linear subspace of $\mathcal{H}$. Let $\mathcal{H}=\bigoplus_{j}H_j$ be a decompositon of $\mathcal{H}$ in orthogonal subspaces $H_j$ of dimension $d_{j}\in\mathbb{N}.$ Let $A,B:\mathcal{H}^\infty\rightarrow \mathcal{H}$ be Fourier multipliers relative to the decomposition $\{H_j\}_{j\in \mathbb{N}_0}.$    
\begin{itemize}
        \item[(1)] Assume that the Sch\"rodinger equation
    \begin{equation}\label{MainII:Sc} \begin{cases}\displaystyle i\frac{d u}{dt}=Au+Bv,& 
    \\
\\u(0)=u_{0}\in \mathcal{H}^\infty,
\text{ } \end{cases}
\end{equation}
    is controllable, then for any $\ell\in \mathbb{N}_0,$ the global symbols $\sigma_A(\ell)$ and $\sigma_B(\ell)$ of $A$ and $B,$ respectively, satisfy the Kalman condition:
    \begin{equation}\label{Rank:Hilbert:spaces:Sc}
     \textnormal{rank}\left[\sigma_B(\ell),\ \sigma_A(\ell)\sigma_B(\ell),\ \cdots,\ \sigma_A(\ell)^{{d_\ell}-1}\sigma_B(\ell)\right]=d_{\ell}.   
    \end{equation} Additionally, if $-iA$ generates a strongly continuous semigroup on $\mathcal{H},$ the image of the Cauchy problem \eqref{Schrodinger:problem:2} under the Fourier transform relative to the decomposition $(H_j)_{j\in \mathbb{N}},$  has a {\it finite global controllability cost} at time $T>0,$ that is 
    $$  \mathscr{C}_{T}:=   \sup_{\ell\in \mathbb{N}_0 }\mathscr{C}_{\ell,T}<\infty.$$
    Moreover, 
    $$   \mathscr{C}_{T}\leq \Tilde{\mathscr{C}}_T, $$ where  $\tilde{\mathscr{C}_{T}}$ is  the controllability cost of \eqref{Schrodinger:problem:2}.
    \item[(2)] Conversely, assume that $-iA$ is the generator of a strongly continuous semigroup on $\mathcal{H},$ and that the Kalman condition \eqref{Rank:Hilbert:spaces:Sc} is satisfied for each $\ell\in\mathbb{N}_0$. Assume that the image of the Cauchy problem \eqref{Schrodinger:problem:2} under the Fourier transform relative to the decomposition $(H_j)_{j\in \mathbb{N}},$  has a {\it finite global controllability cost} in time $T>0,$ that is, $$  \mathscr{C}_{T}:=   \sup_{\ell\in \mathbb{N}_0 }\mathscr{C}_{\ell,T}<\infty. $$ Then, the Cauchy problem  \eqref{MainII:Sc} is controllable at time $T>0,$ and its controllability costs $\tilde{\mathscr{C}_{T}}$ satisfies the inequality
    \begin{equation}
      \mathscr{C}_{T}\geq   \tilde{\mathscr{C}_{T}}.
    \end{equation}
    \end{itemize}
\end{corollary}

\section{Conclusions}\label{conclusions}

In this work, we have considered the problem of the controllability of the Cauchy problem on complex  Hilbert spaces. Our approach reduces the controllability of the system to the validity of the Kalman condition for an infinite number of finite-dimensional controllability systems. This reduction is done by the Fourier analysis induced by a fixed orthogonal decomposition of the underlying Hilbert space over subspaces of finite dimension and the criterion is presented in terms of the matrix-valued symbols relative to these kinds of decomposition as developed in \cite{FourierHilbert,DelRuzTrace1}. After presenting our main Theorem  \ref{MainII}
we have identified in Algorithm \ref{algor} the required steps to analyse the controllability for a variety of problems that satisfy the invariance property in Theorem  \ref{MainII}. In particular, we have introduced the notion of the {\it global controllability cost} of the image of a system under the Fourier transform relative to a decomposition $(H_j)_{j\in \mathbb{N}}$ of a Hilbert space $\mathcal{H}.$ In terms of this notion we have estimated in a sharp way the {\it controllability cost} of the system. The prototype of the models under consideration as well as their controllability has been extensively analysed in Section \ref{Applications}.  There we have considered the control of subelliptic diffusion models on compact Lie groups associated with left-invariant operators, the control of fractional diffusion models for elliptic operators on compact manifolds and also, we have deduced some properties of the control for wave and Schr\"odinger equations and we have explained/illustrated the relation of such properties with respect to Kalman type criteria.

\bibliographystyle{amsplain}

\end{document}